%% file: paper_dr_convergence_amo.tex
\documentclass{article}

\usepackage{amsmath}
\usepackage[utf8]{inputenc}
\usepackage{amssymb}
\usepackage{booktabs}
\usepackage{graphicx}
\usepackage{xcolor}
\usepackage{subfigure}
\usepackage{hyperref}

\usepackage{amsthm}

\theoremstyle{plain}
\newtheorem{theorem}{Theorem}[section]

\newtheorem{lemma}{Lemma}[section]

\theoremstyle{definition}

\newtheorem{remark}{Remark}[section]

\newcommand{\R}{\mathbb{R}}
\newcommand{\Rext}{\mathbb{R}\cup\{+\infty\}}
\newcommand{\Hbs}{\mathcal{H}}
\newcommand{\Ball}{\mathbb{B}}
\newcommand{\C}{\mathbb{C}}
\newcommand{\Id}{\mathbb{I}}
\newcommand{\abs}[1]{\vert#1\vert}
\newcommand{\set}[1]{\left\{#1\right\}}
\newcommand{\norm}[1]{\Vert#1\Vert}
\newcommand{\multito}{\rightrightarrows}
\newcommand{\weakto}{\rightharpoonup}
\DeclareMathOperator*{\argmin}{argmin}
\DeclareMathOperator{\proj}{proj}

\DeclareMathOperator{\sign}{sign}

\DeclareMathOperator{\real}{Re}
\DeclareMathOperator{\rank}{rank}

\newcommand{\Eproof}{\hfill $\square$}
\newcommand{\iprod}[1]{\left\langle #1\right\rangle}
\newcommand{\iprods}[1]{\langle #1\rangle}

\title{Non-stationary Douglas-Rachford and alternating direction method of  multipliers: adaptive stepsizes and convergence\thanks{This material was based upon work partially supported by the National Science Foundation under Grant DMS-1127914 to the Statistical and Applied Mathematical Sciences Institute. 
Any opinions, findings, and conclusions or recommendations expressed in this material are those of the author(s) and do not necessarily reflect the views of the National Science Foundation.
The work of Q. Tran-Dinh was partially supported by the NSF grant, no. DMS-161984.
}}

\author{Dirk A. Lorenz
   \thanks{Institute of Analysis and Algebra, TU Braunschweig, 38092 Braunschweig, Germany,
   (\href{mailto:d.lorenz@tu-braunschweig.de}{d.lorenz@tu-braunschweig.de}).}
  \and Quoc Tran-Dinh
  \thanks{Department of Statistics and Operations Research,
   University of North Carolina at Chapel Hill,
   Chapel Hill, NC (\href{mailto:quoctd@email.unc.edu}{quoctd@email.unc.edu}).}
}

\begin{document}
\maketitle

\begin{abstract}
We revisit the classical Douglas-Rachford (DR) method for finding a zero of the  sum of two maximal monotone operators.  
Since the practical performance of  the DR method  crucially depends on the stepsizes, we aim at developing an adaptive stepsize rule.  
To that end, we  take a closer look at a linear case of the problem and use our  findings to develop a stepsize strategy that eliminates the need for stepsize tuning.  
We analyze a general non-stationary DR scheme  and prove its convergence for a convergent sequence of stepsizes with summable increments in the case of maximally monotone operators.
This, in turn, proves the convergence of the method with the new adaptive stepsize rule.  
We also derive the related non-stationary alternating direction method of multipliers (ADMM)
We illustrate the efficiency of the proposed methods on several numerical examples.

\end{abstract}


\textbf{AMS:}
    90C25, 65K05, 65J15, 47H05
% 90C25 Convex programming 
% 65K05 Mathematical programming methods 
% 65J15 Equations with nonlinear operators 
% 47H05 Monotone operators and generalizations 

%%%%%%%%%%%%%%%%%%%%%%%%%%%%%%%%%%%%%%%%%%%%%%%%%
\section{Introduction}\label{sec:intro}
%%%%%%%%%%%%%%%%%%%%%%%%%%%%%%%%%%%%%%%%%%%%%%%%%
In this paper we consider the Douglas-Rachford (DR) method to solve the problem of finding a zero of the sum of two maximal monotone operators, i.e., solving:
\begin{equation}\label{eq:mono_inc}
0\in (A+B)x,
\end{equation}
where $A,B : \Hbs \multito \Hbs$ are two (possibly multivalued) maximal monotone operators from a Hilbert space $\Hbs$ into itself~\cite{rockafellar1976monotone}.

The DR method originated from~\cite{douglas1956numerical} and was initially proposed to solve the discretization of stationary and non-stationary heat equations 
where the involved monotone operators are linear (namely, the discretization of second derivatives in different spatial directions, 
for example $A\approx -\partial_{x}^{2}$ and $B\approx -\partial_{y}^{2}$). 
The iteration uses resolvents $J_{A} = (\Id + A)^{-1}$ ($\Id$ is the identity map) and $J_B = (\Id + B)^{-1}$, and from the original paper~\cite[Eq. (7.4), (7.5)]{douglas1956numerical} one can extract the iteration
\begin{equation}\label{eq:DR-iter-1}
u^{n+1} := J_{tB}\left( J_{tA}((\Id - tB)u^{n}) + tBu^{n} \right),
\end{equation}
where $t > 0$ is a given stepsize. 
This iteration scheme also makes sense for general maximal monotone operators as soon as $B$ is single-valued.
It has been observed in~\cite{lions1979splitting} that the iteration can be rewritten for arbitrary maximally monotone operators by substituting $u := J_{tB}y$ and using the identity
\begin{equation}
  tBJ_{tB}y = tB(\Id + tB)^{-1}y = y - (\Id + tB)^{-1}y = y - J_{tB}y\label{eq:resolvent-By}
\end{equation}
 to get
\begin{equation}\label{eq:DR-iter-2}
 y^{n+1} := y^n + J_{tA}\left( 2J_{tB}y^{n} - y^{n} \right) - J_{tB}y^{n}.
\end{equation}
Comparing \eqref{eq:DR-iter-1} and \eqref{eq:DR-iter-2}, we see that \eqref{eq:DR-iter-2} does not require to evaluate $Bu$, which avoids assuming that $B$ is single-valued as in \eqref{eq:DR-iter-1}. 
Otherwise, $u^{n+1}$ is not uniquely defined.
While $\set{u^n}$ in \eqref{eq:DR-iter-1} converges to a solution $x^{\ast}$ of \eqref{eq:mono_inc}, $\set{y^n}$ in \eqref{eq:DR-iter-2} is just an intermediate sequence converging to $y^{\ast}$ such that $u^{\ast} = (\Id + tB)^{-1}y^{\ast}$ is a solution of \eqref{eq:mono_inc}. Therefore, \eqref{eq:DR-iter-1} gives us a convenient form to study the DR method in the framework of fixed-point theory.
Note that the iterations~\eqref{eq:DR-iter-1} and~\eqref{eq:DR-iter-2} are equivalent in the stationary case, but they are not equivalent in the non-stationary case, i.e., when the stepsize $t$ varies along the iterations; we will shed more light on this later in Section~\ref{subsec:DR_scheme2}.

From a practical point of view, the performance of a DR scheme mainly depends on the following two aspects:
\begin{itemize}
\item \textit{Good stepsize $t$:}
It seems to be generally acknowledged that the choice of the stepsize is crucial for the algorithmic performance of the method but a general rule to choose the stepsize seems to be missing \cite{bauschke2017convex,eckstein1992douglas}.
So far, convergence theory of DR methods provides some theoretical guidance to select the parameter $t$ in a given range in order to guarantee convergence of the method.
Such a choice is often globally fixed for all iterations, and does not take into account  local structures of the underlying operators $A$ and $B$.
Moreover, the global convergence rate of the DR method is known to be $\mathcal{O}(1/n)$ under only monotonicity assumption, but often using an averaging sequence \cite{Davis2014,davis2016convergence,he20121}, where $n$ is the iteration counter.
Several experiments have shown that DR methods have better practical rate than its theoretical bound \cite{patrinos2014douglas} by using the last iterate (i.e. not an averaging sequence).
In the special case of convex optimization problems, the Douglas-Rachford method is equivalent to the alternating direction methods of multipliers (ADMM) (see, e.g. the recent~\cite{glowinski2014alternating} for a short historical account) and there a several proposals for dynamic stepsizes for ADMM~\cite{he2000alternating,lin2011linearized,song2016fast,xu2016adaptive} but we are not aware of a method that applies to DR in the case of monotone operators.
The recent work~\cite{moursi2018douglas} provides explicit choices for constant stepsizes in cases where the monotone operator posses further properties.

\item \textit{Proper metric:}
Since the DR method is not invariant as the Newton method, the choice of metric and preconditioning seems to be crucial to accelerate its performance.
Some researchers have been studying this aspect from different views, see, e.g., \cite{bredies2016accelerated,bredies2015preconditioned,giselsson2014diagonal,giselsson2017linear,he2012convergence,pock2011diagonal}.
Clearly, the choice of a metric and a preconditioner also affects the choice of the stepsize. 

\end{itemize}
%\toadd{Check  literature}

Note that a metric choice often depends on the variant of methods, while the choice of stepsize depends on the problem structures such as the strongly monotonicity parameters and the Lipschitz constants~\cite{moursi2018douglas}.
In general cases, where $A$ and $B$ are only monotone, we only have a general rule to select the parameter $t$ to obtain its sublinear convergence rate \cite{davis2016convergence,eckstein1992douglas,he20121}.
This stepsize depends on the mentioned global parameters only and does not adequately capture the local structure of $A$ and $B$ to adaptively update $t$.
For instance, a linesearch procedure to evaluate a local Lipschitz constant for computing stepsize in first-order methods can beat the optimal stepsize using  global Lipschitz constant \cite{Becker2011a}, or a Barzilai-Borwein stepsize in gradient descent methods essentially exploits local curvature of the objective function to obtain a good performance.

\textbf{Our contribution:}
We prove the convergence of a new version of the non-stationary Douglas-Rachford method for the case where both operators are merely maximally monotone.
Moreover, we propose a very simple adaptive step-size rule and demonstrate that this rule does improve convergence in practical situations.
We also transfer our results to the case of ADMM and obtain a new adaptive rule that outperforms previously known adaptive ADMM methods and also does have a convergence guarantee.
Our step-size rule is relatively simple and does not incur significantly computational effort rather than the norm of two vectors. Our stepsize rule has a theoretical convergence guarantee.
% While we will not give a general stepsize rule for $t$ in this paper, we will propose a way of adapting the stepsize $t$ dynamically and prove that the method will still converge. 
% Our stepsize is not completely a ``trial-and-error'' quantity. 
% It is based on an observation from analyzing the spectral radius of the DR operator in the linear case (a similar study with different focus can be found in~\cite{bauschke2017affineDR}). 
% Then, we generalize this rule to the nonlinear case when $B$ is single-valued. 
% When $B$ is not single-valued, we also provide an update rule derived from the same observation. 
% Especially, we are able to analyze the convergence of a new variant of the non-stationary DR method that differs from the non-stationary method  that has been analyzed in~\cite{Liang2017}.
% Since DR is equivalent to ADMM in the dual setting, we also derive a similar step-size for ADMM.
% The adaptive stepsize for ADMM derived from the non-stationary DR method seems to be new.

\textbf{Paper organization:}
We begin with an analysis of the case of linear monotone operators in section~\ref{sec:analysis-linear}, analyze the convergence of the non-stationary form of the iteration~\eqref{eq:DR-iter-1}, i.e. the form where $t = t_{n}$ varies with $n$, in section~\ref{sec:conv-non-stationary-dr}, and then propose adaptive stepsize rules in section~\ref{sec:adaptive-stepsize}. 
Section~\ref{sec:admm} extends the analysis to non-stationary ADMM. 
Finally, section~\ref{sec:numerical-experiments} provides several numerical experiments for the DR scheme and ADMM using our new stepsize rule.

%%% 1.1. State of the art.
\subsection{State of the art}\label{sec:state-ofthe-art}
There are several results on the convergence of the iteration~\eqref{eq:DR-iter-2}. 
The seminal paper~\cite{lions1979splitting} showed that, for any positive stepsize $t$, the iteration map
in~\eqref{eq:DR-iter-2} is firmly nonexpansive, that the sequence $\set{y^{n}}$ weakly converges  to a fixed point of the iteration
map~\cite[Prop. 2]{lions1979splitting} and that, $\set{u^{n} = J_{tB}y^{n}}$ weakly converges  to a solution of the inclusion \eqref{eq:mono_inc} as soon as $A$, $B$ and $A+B$ are maximal monotone~\cite[Theorem 1]{lions1979splitting}.  
In the case where $B$ is coercive and Lipschitz continuous, linear convergence was also shown in \cite[Proposition 4]{lions1979splitting}.
These results have been extended in various ways.
Let us attempt to summarize some key contributions on the DR method.
Eckstein and Bertsekas in \cite{eckstein1992douglas} showed that the DR scheme can be cast into a special case of the proximal point method \cite{rockafellar1976monotone}.
This allows the authors to exploit inexact computation from the proximal point method \cite{rockafellar1976monotone}. 
They also presented a connection between the DR method and the alternating direction method of multipliers (ADMM).
In \cite{svaiter2011weak} Svaiter proved a weak convergence of the DR method in Hilbert space without the assumption that $A+B$ is maximal monotone and the prove have been simplified in~\cite{svaiter2018simplified}.
Combettes \cite{combettes2004solving} cast the DR method as special case of the averaging operator from a fixed-point framework.
Applications of the DR method have been studied in \cite{combettes2007douglas}.
The convergence rate of the DR method was first studied in \cite{lions1979splitting} for the strongly monotone case, while the sublinear rate was then proved in \cite{he20121}.
A more intensive research on convergence rates of the DR methods can be found in \cite{Davis2014,davis2016convergence,moursi2018douglas,nishihara2015general}.
The DR method has been extended to accelerated variant in \cite{patrinos2014douglas} but specifying for a special setting.
In~\cite{Liang2017} the authors analyzed a non-stationary DR method derived from~\eqref{eq:DR-iter-2} in the framework of perturbations of non-expansive iterations and showed convergence for convergent stepsize sequences with summable errors.

The DR method together with its dual variant, ADMM, become extremely popular in recent years due to a wide range of applications in image processing, and machine learning \cite{becker2013algorithm,li2015proximal}, which are unnecessary to recall them all here.

In terms of stepsize selection for DR schemes as well as for ADMM methods, it seems that there is very little work available from the literature. 
Some general rules for fixed stepsizes based on further properties of the operators such as strong monotonicity, Lipschitz continuity, and coercivity are given in~\cite{giselsson2017tightglobalrates,moursi2018douglas}, and it is shown that the resulting linear rates are tight.
Heuristic rules for fixed stepsizes motivated by quadratic problems are derived in~\cite{giselsson2017linear}.
A self-adaptive stepsize for ADMM proposed in \cite{he2000alternating} seems to be one of the first work in this direction.
The recent works~\cite{xu2016adaptive,xu2017adaptive} also proposed an adaptive update rule for stepsize in ADMM based on a spectral estimation.
Some other papers rely on theoretical analysis to choose optimal stepsize such as \cite{ghadimi2015optimal}, but it only works in the quadratic case.
In \cite{lin2011linearized}, the authors proposed a nonincreasing adaptive rule for the penalty parameter in ADMM. 
Another update rule for ADMM can be found in \cite{song2016fast}.
While ADMM is a dual variant of the DR scheme, we unfortunately have not seen any work that converts such an adaptive step-size from ADMM to the DR scheme where the more general case of monotone operators can be handled.
In addition, the adaptive step-size for the DR scheme by itself seems to not exist in the literature.

%%% 1.2. A motivating example.
\subsection{A motivating linear example}
\label{sec:motivating-example}
While the Douglas-Rachford iteration (weakly) converges  for all positive stepsizes $t > 0$, it seems to be folk wisdom, that there is a ``sweet spot'' for the stepsize which leads to fast convergence. 
We illustrate this effect with a simple linear example. 
We consider a linear equation
\begin{equation}\label{eq:linear_exam}
0 = Ax+Bx,
\end{equation}
where $A,B\in\R^{m\times m}$ are two matrices of the size $m\times m$ with $m=200$. 
We choose symmetric positive definite matrices with $\rank(A) = \tfrac{m}2+10$ and $\rank{B} = \tfrac{m}2$ such that $A+B$ has full rank, 
and thus the equation $0 = Ax + Bx$ has zero as its unique solution.\footnote{The exact construction of $A$ and $B$ is $A = C^{T}C$ and $B = D^{T}D$, where $C\in\R^{(0.5m + 10)\times m}$ and $D\in\R^{0.5m\times m}$ are drawn from the standard Gaussian distribution in Matlab.} 
Since $B$ is single-valued, we directly use the iteration~\eqref{eq:DR-iter-1}.

%%% Remark 1.1.
\begin{remark}\label{rem:linear-shift}
Note that the shift $\tilde B x= Bx-y$ would allow to treat inhomogeneous equation $(A+B)x = y$. 
If $x^{*}$ is a solution of this equation, then one sees that iteration~\eqref{eq:DR-iter-1} applied to $A+\tilde B$ is equivalent to applying the iteration to $A+B$ but for the residual $x-x^{*}$.
\end{remark}

We ran the DR scheme \eqref{eq:DR-iter-1} for a given range of different values of $t > 0$, and show the residuals $\norm{(A+B)u^{n}}$ in semi-log-scale on the left of Figure~\ref{fig:motivating-example-residual}.  
One observes the following typical behavior for this example: 
\begin{itemize}
\item A not too small stepsize ($t=0.5$ in this case) leads to good progress in the beginning, but slows down considerably in the end. 
\item Large stepsizes (larger than $2$ in this case) are slower in the beginning and tend to produce non-monotone decrease of the error.
\item Somewhere in the middle, there is a stepsize which performs much better than the small and large stepsizes.
\end{itemize}
In this particular example the stepsize $t=1.5$  greatly outperforms the other stepsizes.
On the right of Figure~\ref{fig:motivating-example-residual} we show the norm of the residual after a fixed number of iterations for varying stepsizes.
One can see that there is indeed a sweet spot for the stepsizes around $t=1.5$.
Note that the value of $t=1.5$ is by no means universal and this sweet spot of $1.5$ varies with the problem size, with the ranks of $A$ and $B$, and even for each particular instance of this linear example.

\begin{figure}[htp!] 
\begin{center}
 \includegraphics[width=0.49\textwidth]{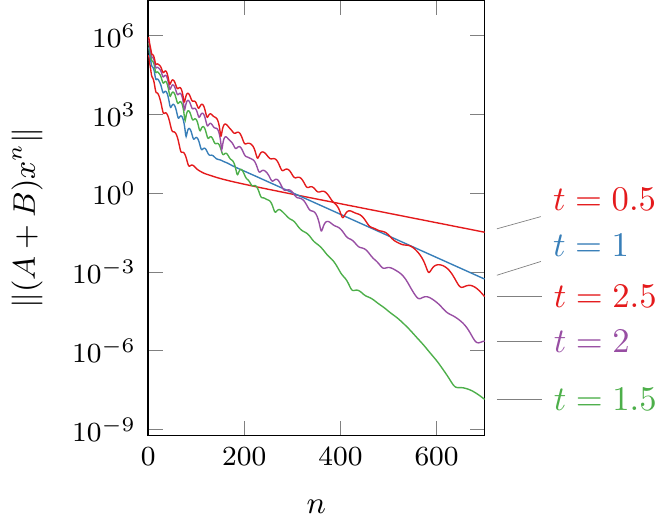}~%
 \includegraphics[width=0.49\textwidth]{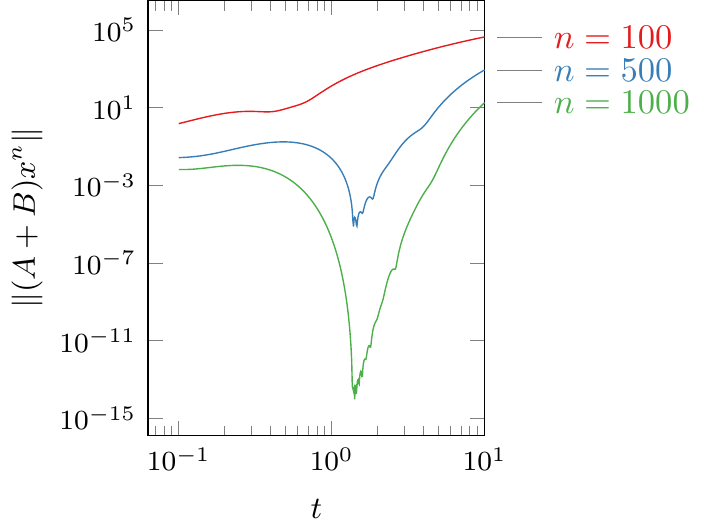}
 \vspace{-1ex}
 \caption{The residual of the Douglas-Rachford scheme \eqref{eq:DR-iter-1} for the linear example. 
 Left: The dependence of the residual on the iterations with different stepsizes. 
 Right: The dependence of the residual on the the stepsize with different numbers of iterations.}
 \label{fig:motivating-example-residual}
 \end{center}
 \vspace{-2ex}
\end{figure}

%%%%%%%%%%%%%%%%%%%%%%%%%%%%%%%%%%%%%%%%%%%%%%%%%
%%% 2. Analysis of the linear case
%%%%%%%%%%%%%%%%%%%%%%%%%%%%%%%%%%%%%%%%%%%%%%%%%
\section{Analysis of the linear monotone inclusion}\label{sec:analysis-linear}
In order to develop an adaptive stepsize for our non-stationary DR method, we first consider the linear problem instance of \eqref{eq:mono_inc}.
We consider the original DR scheme \eqref{eq:DR-iter-1} instead of \eqref{eq:DR-iter-2} since \eqref{eq:DR-iter-1} generates the sequence $\set{u^n}$ which converges to a solution of \eqref{eq:mono_inc}, while the sequence $\set{y^n}$ computed by \eqref{eq:DR-iter-2} does not converge to a solution and  its limit does depend on the stepsize in general.

\subsection{The construction of adaptive stepsize for single-valued operator $B$}
\label{subsec:adaptive-1}
When both $A$ and $B$ are linear, the DR scheme~\eqref{eq:DR-iter-1} can be expressed as a fixed-point iteration scheme of the following mapping:
\begin{equation}\label{eq:linear-iteration-map}
\begin{split}
H_{t} & := J_{tB}\left( J_{tA}(\Id - tB) + tB\right)\\
         & = (\Id+tB)^{-1}(\Id + tA)^{-1}(\Id + t^{2}AB)\\
         & = (\Id + tA + tB + t^{2}AB)^{-1}(\Id + t^{2}AB).    
\end{split}
\end{equation}
Recall that, by Remark~\ref{rem:linear-shift}, all of this section also applies not only to problem $(A+B)x=0$ but also problem $(A+B)x = y$.
The notion of a monotone operator has a natural equivalence for matrices, which is, however, not widely used. 
Hence, we recall that
a matrix $A\in\R^{m\times m}$ is called \emph{monotone}, if, for all $x\in\R^{m}$, it holds that $\iprod{x,Ax} \geq 0$.
Note that any symmetric positive semidefinite (spd) matrix is monotone, but a monotone matrix is not necessarily spd.
Examples of a monotone matrices that are not spd are
\begin{equation*}
A = \begin{bmatrix}
  0 & -1\\
  1 & 0
\end{bmatrix},\quad\text{and}\quad
A = \begin{bmatrix}
  1 & t\\
  0 & 1
\end{bmatrix}\ 
~~\text{with}~\abs{t}\leq 2.
\end{equation*}
The first matrix is skew symmetric, i.e., $A^T = -A$ and any such matrix is monotone.
Note that even if $A$ and $B$ are spd (as in our example in Section~\ref{sec:motivating-example}), the iteration map $H_{t}$ in \eqref{eq:linear-iteration-map} is not even symmetric.
Consequently, the asymptotic convergence rate of the iteration scheme~\eqref{eq:DR-iter-1} is not governed by the norm of $H_{t}$ but by its \emph{spectral radius} $\rho(H_{t})$, which is the largest magnitude of an eigenvalue of $H_{t}$ (cf.~\cite[Theorem 11.2.1]{golub2013matrix}). 
Moreover, the eigenvalues and eigenvectors of $H_{t}$ are complex in general. 

First, it is clear from the derivation of $H_{t}$ that the eigenspace of $H_t$ for the eigenvalue $\lambda=1$ exactly consists of the solutions of $(A+B)x = 0$.

In the following, for any $z\in\C$ (the set of complex numbers) and $r > 0$, we denote by $\Ball_{r}(z)$ the ball of radius $r$ centered at $z$.
We estimate the eigenvalues of $H_{t}$ that are different from $1$.

%%% Lemma 2.2.
\begin{lemma}\label{lemma:estimate-ev-Ht}
Let $A,B\in\R^{n\times n}$ be monotone, and $H_{t}$ be defined by~\eqref{eq:linear-iteration-map}. 
Let $\lambda\in \C$ be an eigenvalue of $H_{t}$ with corresponding eigenvector $z\in\C^{n}$. 
Assume that $\lambda\neq 1$ and define $c$ by
\begin{equation}\label{eq:c-estimate-evs}
c := \frac{\real(\iprod{Bz,z})}{t^{-1}\norm{z}^{2} + t\norm{Bz}^{2}}.
\end{equation}
Then, we have $c\geq 0$ and 
\begin{equation*}
\left\vert{\lambda-\frac12}\right\vert \leq  \sqrt{\frac14 - \frac{c}{1+2c}}\leq \frac12,
\end{equation*}
i.e. $\lambda\in \Ball_{\frac12}\left(\frac12\right)$.
\end{lemma}
\begin{proof}
Note that for a real, linear, and monotone map $M$, and a complex vector $a = b+ic$, it holds that $\iprod{Ma,a} = \iprod{Mb,b} + \iprod{Mc,c} + i\iprod{(M^{T}-M)b, c}$ and thus, $\real(\iprod{Ma,a})\geq 0$. This shows that $c\geq 0$.

We can see from \eqref{eq:linear-iteration-map} that any pair $(\lambda,z)$ of eigenvalue and eigenvector of $H_t$ fulfills
\begin{equation*}
  z+t^{2}ABz = \lambda(z+tAz+tBz+t^{2}ABz).
\end{equation*}
Now, if we denote $u := Bz$, then this expression becomes
\begin{equation*}
z + t^{2}Au = \lambda z + \lambda tAz + \lambda t u + \lambda t^{2}Au,
\end{equation*}
 which, by rearranging, leads to
\begin{equation*}
  -(\lambda-1)z - \lambda t u = tA(\lambda z + (\lambda-1)tu).
\end{equation*}
Hence, by monotonicity of $tA$, we can derive from the above relation that
\begin{align*}
0 & \leq \real(\iprod{\lambda z + (\lambda-1)tu,-(\lambda-1)z - \lambda tu}\label{eq:A-monotone})\\
   & = -\real(\lambda(\bar\lambda-1))\norm{z}^{2} - (\abs{\lambda}^{2} + \abs{\lambda-1}^{2})t\real(\iprod{u,z}) - \real((\lambda-1)\bar\lambda) t^{2}\norm{u}^{2}.
\end{align*}
This leads to
\begin{equation*}
(\abs{\lambda}^{2} + \abs{\lambda-1}^{2})\real(\iprod{u,z}) \leq \frac{\real(\lambda - \abs{\lambda}^{2})}{t} \norm{z}^{2} + \real((\bar\lambda - \abs{\lambda}^{2}))t\norm{u}^{2}.
\end{equation*}
Denoting $\lambda := x+iy \in \C$, the last expression reads as 
\begin{equation*}
(x^{2}+(x-1)^{2}+2y^{2})\real(\iprod{u,z})\leq (x-x^{2}-y^{2})\Big(\frac{\norm{z}^{2}}t + t\norm{u}^{2}\Big).
\end{equation*}
Recalling the definition of $c$ in \eqref{eq:c-estimate-evs}, we get 
\begin{equation*}
(x^{2}+(x-1)^{2}+2y^{2}) c\leq x-x^{2}-y^{2}.
\end{equation*}
This is equivalent to
\begin{align*}
0  \leq x-x^{2}-y^{2} - cx^{2}-c(x-1)^{2}-2cy^{2} = (1+2c)(x-x^{2}-y^{2}) - c,
\end{align*}
which, in turn, is equivalent to
$%\begin{equation*}
x^{2}-x + y^{2}\leq - \tfrac{c}{1+2c}$.
%\end{equation*}
Adding $\tfrac{1}{4}$ to both sides, it leads to
$%\begin{equation*}
(x-\tfrac12)^{2}+y^{2}\leq \tfrac14 - \tfrac{c}{1+2c}$,
%\end{equation*}
which shows the desired estimate.
\end{proof}
%%% End of the proof.

In general, the eigenvalues of $H_{t}$ depend on $t$ in a complicated way. 
For $t=0$, we have $H_{0} =  \Id$ and hence, all eigenvalues are equal to one. 
For growing $t > 0$, some eigenvalues move into the interior of the circle $\Ball_{1/2}(1/2)$ and for $t\to \infty$, it seems that all eigenvalues tend to converge to the boundary of such a circle, see Figure~\ref{fig:evs-estimate-illustration} for an illustration of eigenvalue distribution.

\begin{figure}[htb]
\begin{center}
\includegraphics[width=0.5\textwidth]{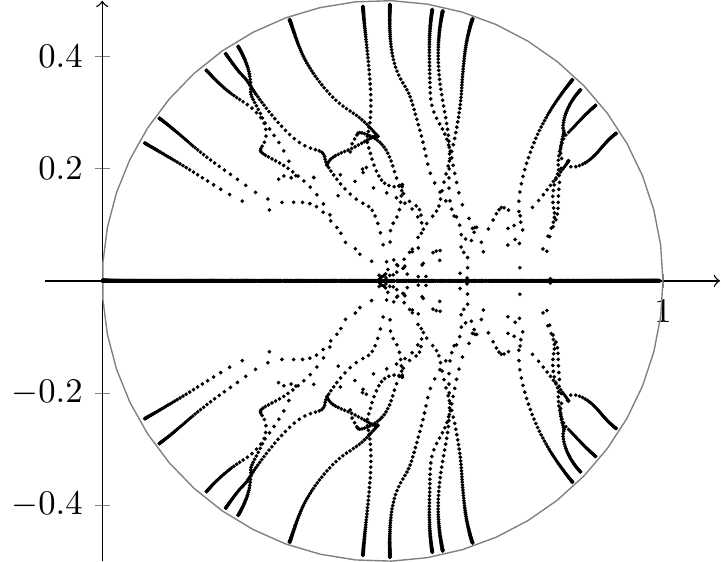}
\caption{Eigenvalues of $H_t$ for different values of $t$ for a linear example similar to the example \eqref{eq:linear_exam} in Section~\ref{sec:motivating-example} $($but with $m=50$$)$.}\label{fig:evs-estimate-illustration}
\end{center}
\vspace{-2ex}
\end{figure}

\begin{remark}\label{re:firmlynonexpansiveness}
It appears that Lemma~\ref{lemma:estimate-ev-Ht} is related to
Proposition 4.10 of~\cite{Bauschke2012} and also to the fact that
the iteration mapping $H_{t}$ is (in the general nonlinear case) known to be not only non-expansive, but \emph{firmly non-expansive}~(cf.~\cite[Lemma 1]{eckstein1992douglas} and \cite[Figure 1]{eckstein1992douglas}).  
In general, firmly non-expansiveness allows over-relaxation of the method, and
indeed, one can also easily see this in the linear case as well: 
If $\lambda$ is an eigenvalue of $H_{t}$, then it lies in $\Ball_{1/2}(1/2)$ (when it is not equal to one) and the corresponding eigenvalue $\lambda_{\rho}$ of the relaxed   iteration map
\begin{equation*}
H_{t}^{\rho} = (1-\rho)\Id + \rho H_{t}
\end{equation*}
is $\lambda_{\rho} = 1-\rho + \rho\lambda$ and lies in $\Ball_{\rho/2}(1-\tfrac\rho2)$. 
Therefore, for  $0\leq \rho\leq 2$ all eigenvalues different from one of the  relaxed iteration
\begin{equation*}
u^{n+1} = (1-\rho) u^{n} + \rho H_{t}u^{n}
\end{equation*}
lie in a circle of radius $\rho/2$ centered at $1-\rho/2$, and hence, the iteration  is still non-expansive. 
It is know that relaxation can speed up convergence, but we will not investigate this in this paper.
\end{remark}

Lemma \ref{lemma:estimate-ev-Ht} tells us a little more than that all eigenvalues of the iteration map $H_{t}$ lie in a circle centered at $\tfrac12$ of
radius $\tfrac12$. 
Especially, all eigenvalues except for $\lambda=1$ have
magnitude strictly smaller than one if $\real(\iprod{Bz,z})>0$ for all corresponding eigenvectors $z$. This implies
that the iteration map $H_{t}$ is indeed asymptotically contracting outside the set of solutions $\set{x^{\ast} \in\Hbs \mid (A+B)x^{\ast} =0}$ of \eqref{eq:mono_inc}. This proves that the
stationary iteration $u^{n+1} = H_{t}u^{n}$ converges to a zero point of the map $A+B$ at a linear rate.
Note that this does not imply the convergence in the non-stationary case.

To optimize the convergence speed, we aim at minimizing the spectral radius of $H_t$, which is the magnitude of the largest eigenvalue of $H_{t}$ and there seems to be little hope to  explicitly minimize this quantity.

Here is a heuristic argument based on Lemma~\ref{lemma:estimate-ev-Ht}, which we will use to derive an adaptive stepsize rule: 
Note that $c\mapsto \tfrac{c}{1+2c}$ is increasing and hence, to minimize the upper bound on $\lambda$ (more precisely: the distance of $\lambda$ to $\tfrac12$) we want to make $c$ from~\eqref{eq:c-estimate-evs} as large as
possible. This is achieved by minimizing the denominator of $c$ over $t$ which happens for
\begin{equation*}
t = \frac{\norm{z}}{\norm{Bz}}.
\end{equation*}
This gives $c = \real(\iprod{Bz,z})/(2\norm{z}\norm{Bz})$ and note that $0\leq c\leq 2$ (which implies $0\leq \tfrac{c}{1+2c}\leq\tfrac14$). 
This motivates an adaptive choice for the stepsize $t_n$ as
\begin{equation}\label{eq:adaptive_stepsize}
t_{n} := \frac{\norm{u^{n}}}{\norm{Bu^{n}}},
\end{equation}
in  the Douglas-Rachford iteration scheme \eqref{eq:DR-iter-1}.  
\begin{remark}
  One can use the above derivation to deduce that $t = 1/\norm{B}$ is a good constant step-size.
  In fact, this is also the stepsize that gives the best linear rate derived in~\cite[Proposition 4]{lions1979splitting}, which is minimized when $t=1/M$ where $M$ is the Lipschitz constant of $B$.
  However, this choice does not perform well in practice in our experiments.
\end{remark}

Since little is known about the non-stationary Douglas-Rachford iteration in general (besides the result from \cite{Liang2017} on convergent stepsizes with summable
errors), we turn to an investigation of this method in Section~\ref{sec:conv-non-stationary-dr}.
Before we do so, we generalize the heuristic stepsize to the case of multivalued $B$.
\subsection{The construction of adaptive stepsize for non-single-valued $B$}\label{subsec:DR_scheme2}
In the case of multi-valued $B$, one needs to apply the iteration~\eqref{eq:DR-iter-2} instead of~\eqref{eq:DR-iter-1}.
To motivate an adaptive choice for the stepsize in this case, we again consider situation of linear operators.

In the linear case, the iteration~\eqref{eq:DR-iter-2} is given by the iteration matrix
\begin{equation*}
F_{t} = J_{tA}(2J_{tB} - \Id) - J_{tB} + \Id.
\end{equation*}
Comparing this with the iteration map $H_{t}$ from~\eqref{eq:linear-iteration-map} (corresponding to~\eqref{eq:DR-iter-1}) one notes that
\begin{equation*}
F_{t} = (\Id+tB) H_{t}(\Id + tB)^{-1},
\end{equation*}
i.e., the matrices $F_{t}$ and $H_{t}$ are similar and hence, have the same eigenvalues.
Moreover, if $z$ is an eigenvector of $H_{t}$ with the eigenvalue $\lambda$, then $(\Id+tB)z$ is an eigenvector of $F_{t}$ for the same eigenvalue $\lambda$.

However, in the case of the iteration~\eqref{eq:DR-iter-2} we do not assume that $B$ is single-valued, and thus, the adaptive stepsize using the quotient $\norm{u}/\norm{Bu}$ cannot be used.
However, again due to~\eqref{eq:resolvent-By}, we can rewrite this quotient without applying $B$ and get, with $J_{tB}y = u$, that
\begin{equation}\label{eq:adaptive_stepsize2}
\frac{\norm{u}}{\norm{Bu}} = \frac{\norm{J_{tB}y}}{\norm{\tfrac1t(y-J_{tB}y)}} = t \frac{\norm{J_{tB}y}}{\norm{y-J_{tB}y}}.
\end{equation}

Note that the two iteration schemes \eqref{eq:DR-iter-1}  and~\eqref{eq:DR-iter-2} are not equivalent in the non-stationary  and non-linear case. 
Indeed, let us consider $y^n$ such that $u^n := J_{t_{n-1}B}y^n$.  
By induction, we have $u^{n+1} = J_{t_{n}B}y^{n+1}$.  Substituting $u^{n+1}$ into \eqref{eq:DR-iter-1}, we obtain
\begin{align}\label{eq:yn_point}
	y^{n+1} = J_{t_{n}A}\left(u^n - t_nBu^n\right) + t_nBu^n.
\end{align}
From~\eqref{eq:resolvent-By} we have 
\begin{equation*}
	Bu^n = BJ_{t_{n-1}B}y^n =  \tfrac{1}{t_{n-1}}\left(y^n - J_{t_{n-1}B}y^n\right).
\end{equation*}
Substituting $u^n = J_{t_{n-1}B}y^n$ and $Bu^n$ into \eqref{eq:yn_point}, we obtain
\begin{align*}
    y^{n+1}  &=  J_{t_nA}\left( J_{t_{n-1}B}y^n - \tfrac{t_n}{t_{n-1}}\left(y^n - J_{t_{n-1}B}y^n\right)\right) + \tfrac{t_n}{t_{n-1}}\left(y^n - J_{t_{n-1}B}y^n\right) \nonumber\\
             &= \tfrac{t_n}{t_{n-1}}y^n + J_{t_nA}\left( \left(1 + \tfrac{t_n}{t_{n-1}}\right)J_{t_{n-1}B}y^n - \tfrac{t_n}{t_{n-1}}y^n \right) - \tfrac{t_n}{t_{n-1}}J_{t_{n-1}B}y^n.\label{eq:nonstat-DR}
\end{align*}
Updating $t_n$ by~\eqref{eq:adaptive_stepsize2} would then give
\begin{equation*}
    t_n := \kappa_n t_{n-1},~~~~\text{where}~~~\kappa_n := \frac{\norm{J_{t_{n-1}B}y^n}}{\norm{y^n - J_{t_{n-1}B}y^n}}.
\end{equation*}
In summary, we can write an alternative DR scheme for solving \eqref{eq:mono_inc}  as
\begin{equation}\label{eq:new_DR_scheme}
\left\{\begin{array}{ll}
        u^n &:= J_{t_{n-1}B}y^n, \vspace{1ex}\\
        \kappa_n &:= \frac{\norm{u^n}}{\norm{u^n - y^n}}, \vspace{1ex}\\
        t_n &:= \kappa_nt_{n-1}, \vspace{1ex}\\ 
        v^n &:= J_{t_nA}\left( (1+\kappa_n)u^n - \kappa_ny^n \right),\vspace{1ex}\\
        y^{n+1} &:= v^n + \kappa_n(y^n - u^n).
\end{array}\right.
\end{equation}
This scheme essentially has the same per-iteration complexity as in the standard DR method since the computation of $\kappa_n$ does not significantly increase the cost.

Note that the non-stationary scheme~\eqref{eq:new_DR_scheme} is notably different from the non-stationary scheme derived directly from~\eqref{eq:DR-iter-2} 
(which has been analyzed in~\cite{Liang2017}).
To the best of our knowledge, the scheme~\eqref{eq:new_DR_scheme} is new.

%%%%%%%%%%%%%%%%%%%%%%%%%%%%%%%%%%%%%%%%%%%%%%%%
%%% 3. Convergence of the non-stationary DR method
%%%%%%%%%%%%%%%%%%%%%%%%%%%%%%%%%%%%%%%%%%%%%%%%
%\section{Convergence of the non-stationary DR method}\label{sec:conv-non-stationary-dr}
\section{Convergence of the non-stationary DR method}\label{sec:conv-non-stationary-dr}
In this section, we prove weak convergence of the new non-stationary scheme~\eqref{eq:new_DR_scheme}.
We follow the approach by~\cite{svaiter2018simplified,svaiter2011weak} and restate the DR iteration as follows: 
Given $(u^{0},b^{0})$ such that $b^{0}\in B(u^{0})$ and a sequence $\set{t_{n}}_{n\geq 0}$, at each iteration $n\geq 0$, we iterate
\begin{equation}\label{eq:DR-iter-svaiter}
\left\{\begin{aligned}
    a^{n}& \in A(v^{n}), & v^{n} + t_{n}a^{n} & = u^{n-1}-t_{n}b^{n-1}\\
    b^{n}& \in B(u^{n}), & u^{n} + t_{n}b^{n} & = v^{n} +
    t_{n}b^{n-1} .
\end{aligned}\right.
\end{equation}
Note that, in the case of single-valued $B$, this iteration reduces to
\begin{equation*}
u^{n} = J_{t_{n}B}(J_{t_{n}A}(u^{n-1}-t_{n}Bu^{n-1}) + t_{n}Bu^{n-1}),
\end{equation*}
and this scheme can, as shown in Section~\ref{subsec:DR_scheme2}, be transformed into the non-stationary iteration scheme~\eqref{eq:new_DR_scheme}.

Below are some consequences which we will need in our analysis:
\begin{align}
&u^{n-1}-u^{n}  = t_{n}(a^{n}+b^{n})\label{eq:newDR-aux1}.\\
&t_{n}(b^{n-1}-b^{n})  = u^{n}-v^{n}\label{eq:newDR-aux2}.\\
&u^{n} - v^{n} + t_{n}(a^{n}+b^{n})  = t_{n}(a^{n} + b^{n-1}) = u^{n-1}-v^{n} \label{eq:newDR-aux3}.
\end{align}
Before proving  our convergence result, we state the following lemma.
%% Lemma 3.1.
\begin{lemma}\label{lem:DR-svaiter-aux}
Let $\set{\alpha_n}$, $\set{\beta_n}$, and $\set{\omega_n}$ be three nonnegative sequences, and $\set{\tau_n}$ be a bounded sequence  such that for $n\geq 0$:
\begin{equation*}
0 < \underline{\tau} \leq \tau_{n} \leq \bar{\tau}, ~~\abs{\tau_{n}-\tau_{n+1}}\leq \omega_{n},~~\text{and}~\sum_{n=0}^{\infty} \omega_{n}<\infty. 
\end{equation*}
If  $\alpha_{n-1} + \tau_{n}\beta_{n-1}\geq \alpha_{n}+\tau_{n}\beta_{n}$, then $\set{\alpha_{n}}$ and $\set{\beta_{n}}$ are bounded.
\end{lemma}

\begin{proof}
If $\tau_{n+1}\leq \tau_{n}$, then
\begin{equation*}
	\alpha_{n-1} + \tau_{n}\beta_{n-1}\geq  \alpha_{n}+\tau_{n}\beta_{n} \geq \alpha_{n}+\tau_{n+1}\beta_{n}.
\end{equation*}
If $\tau_{n+1}\geq \tau_{n}$, then $\frac{\tau_n}{\tau_{n+1}} \leq 1$ and
\begin{equation*}
	\alpha_{n-1} + \tau_{n}\beta_{n-1} \geq  \alpha_{n}+\tau_{n}\beta_{n} \geq \tfrac{\tau_n}{\tau_{n+1}}\alpha_n + \tau_n\beta_{n} = \tfrac{\tau_{n}}{\tau_{n+1}}(\alpha_{n} + \tau_{n+1}\beta_{n}).
\end{equation*}
By the assumption that $\tfrac{\tau_{n}}{\tau_{n+1}}\geq 1 - \tfrac{\omega_{n}}{\underline \tau}$ and, without loss of generality, we assume that the latter term is positive (which is fulfilled for $n$ large enough, because $\omega_{n}\to 0$). Thus, in both cases, we can show that
\begin{equation*}
	\alpha_{n-1} + \tau_{n}\beta_{n-1}\geq \left(1 - \tfrac{\omega_{n}}{\underline \tau}\right)\left(\alpha_{n}+\tau_{n+1}\beta_{n} \right).
\end{equation*}
Recursively, we get
\begin{equation*}
	\alpha_{0}+\tau_{1}\beta_{0}\geq \prod_{l=1}^{n}\left(1 - \tfrac{\omega_{l}}{\underline \tau}\right)\left(\alpha_{n}+\tau_{n+1}\beta_{n}\right).
\end{equation*}
Under the assumption $\sum_{n=0}^{\infty}\omega_n < +\infty$, we have $\prod_{l=1}^{n}\left(1 - \tfrac{\omega_{l}}{\underline \tau}\right) \geq M$ for some $M>0$ and all $n \geq 1$.
Then, we have $\alpha_{n}+\tau_{n+1}\beta_{n} \leq \frac{1}{M}\left(\alpha_{0}+\tau_{1}\beta_{0}\right)$.
This shows that $\set{\alpha_{n}+\tau_{n+1}\beta_{n}}$ is bounded. 
Since $\set{\alpha_n}$, $\set{\beta_n}$, and $\set{\tau_n}$ are all nonnegative, it implies that $\set{\alpha_n}$ and $\set{\beta_n}$ are bounded.
\end{proof}

%%% Theorem 4.1.
\begin{theorem}[Convergence of non-stationary DR]\label{th:main_theorem}
Let  $A$ and $B$ be maximally monotone and $\set{t_{n}}$ be a positive sequence such that
\begin{equation*}
0 < \underline{t} \leq t_{n} \leq \bar{t},\quad\sum_{n=0}^{\infty}\abs{t_{n}-t_{n+1}}<\infty\quad\text{and}\quad t_{n}\to t^{*},
\end{equation*} 
where $0 < \underline{t} \leq \bar{t} < +\infty$ are given.
Then, the sequence $\set{(u^n, b^n)}$ generated by the iteration scheme~\eqref{eq:DR-iter-svaiter}  weakly converges  to some $(u^{*},b^{*})$ in the extended solution set $S(A,B) = \{(z,w)\mid w\in B(z),\ -w\in A(z)\}$ of \eqref{eq:mono_inc}, so in particular, $0 \in (A+B)(u^{*})$.
\end{theorem}
\begin{proof}
The proof of this theorem follows the proof of~\cite[Theorem 1]{svaiter2011weak}.
First, we observe that, for any $(u,b)\in S(A,B)$, we have
\begin{align*}
\iprod{u^{n-1}-u^{n},u^{n}-u} & = t_{n}\iprod{a^{n}+b^{n},u^{n}-u} & \text{by~\eqref{eq:newDR-aux1}}\\
                                  & = t_{n}\big[\iprod{a^{n}+b,u^{n}-u} + \iprod{b^{n}-b,u^{n}-u}\big]\\
                                  & \geq t_{n}\iprod{a^{n}+b,u^{n}-u} & \text{$A$ is monotone}\\
                                  & = t_{n}\big[\iprod{a^{n}+b,u^{n}-v^{n}} + \iprod{a^{n}+b,v^{n}-u}\big]\\
                                  & \geq t_{n}\iprod{a^{n}+b,u^{n}-v^{n}}.& \text{$B$ is monotone}
\end{align*}
From this and~\eqref{eq:newDR-aux2} it follows that 
\begin{align*}
\iprod{u^{n-1}-u^{n},u^{n}-u} + t_{n}^{2}\iprod{b^{n-1}-b^{n},b^{n}-b} & \geq t_{n}\iprod{a^{n} +b, u^{n}-v^{n}} \\
& + t_{n}\iprod{u^{n}-v^{n},b^{n}-b}\\
& = t_{n}\iprod{u^{n}-v^{n},a^{n}+b^{n}}.
\end{align*}
Moreover, by~\eqref{eq:newDR-aux1} and~\eqref{eq:newDR-aux2} it holds that
\begin{equation*}
\norm{u^{n-1}-u^{n}}^{2} + t_{n}^{2}\norm{b^{n-1}-b^{n}}^{2} = t_{n}^{2}\norm{a^{n}+b^{n}}^{2} + \norm{u^{n}-v^{n}}^{2},
\end{equation*}
and thus
\begin{align}
\norm{u^{n-1}-u}^{2} &+ t_{n}^{2}\norm{b^{n-1}-b}^{2} = \norm{u^{n-1} - u^{n} + u^{n} - u}^{2} + t_{n}^{2}\norm{b^{n-1} - b^{n} + b^{n} - b}^{2} \nonumber\\
& = \norm{u^{n-1}-u^{n}}^{2} + 2\iprod{u^{n-1}-u^{n},u^{n}-u} + \norm{u^{n}-u}^{2}\nonumber\\
&\quad +t_{n}^{2}\big[ \norm{b^{n-1}-b^{n}}^{2} + 2\iprod{b^{n-1}-b^{n},b^{n}-b} + \norm{b^{n}-b}^{2}\big]\nonumber\\
& \geq t_{n}^{2}\norm{a^{n}+b^{n}}^{2} + \norm{u^{n}-v^{n}}^{2} + 2t_{n}\iprod{u^{n}-v^{n},a^{n}+b^{n}}\nonumber\\
& \quad+ \norm{u^{n}-u}^{2} + t_{n}^{2}\norm{b^{n}-b}^{2}\nonumber\\
& = \norm{u^{n}-u}^{2} + t_{n}^{2}\norm{b^{n}-b}^{2} + \norm{u^{n}-v^{n}+t_{n}(b^{n}+a^{n})}^{2}.\label{eq:newDR-est1}
\end{align}
We see from \eqref{eq:newDR-est1} that
\[
\norm{u^{n-1}-u}^{2} + t_{n}^{2}\norm{b^{n-1}-b}^{2} \geq  \norm{u^{n}-u}^{2} + t_{n}^{2}\norm{b^{n}-b}^{2},
\]
and using Lemma~\ref{lem:DR-svaiter-aux} with $\alpha_{n} = \norm{u^{n}-u}^{2}$, $\tau_{n} = t_{n}^{2}$ and $\beta_{n} = \norm{b^{n}-b}^{2}$, we can conclude that  both sequences $\set{\norm{u^{n}-u}}$ and $\set{\norm{b^{n}-b}}$ are bounded.
  
Again from~\eqref{eq:newDR-est1} we can deduce using~\eqref{eq:newDR-aux3} that
\begin{equation}\label{eq:newDR-est2}
{\!\!\!\!\!}\begin{array}{ll}
\norm{u^{n-1}{\!\!} -u}^{2} + t_{n}^{2}\norm{b^{n-1}{\!\!} -b}^{2} &{\!\!\!\!} \geq \norm{u^{n}-u}^{2} + t_{n}^{2}\norm{b^{n}-b}^{2} + \norm{u^{n-1}-v^{n}}^{2}\vspace{1.2ex}\\
&{\!\!\!\!} = \norm{u^{n}-u}^{2} + t_{n}^{2}\norm{b^{n}-b}^{2} + t_{n}^{2}\norm{a^{n}+b^{n-1}}^{2}.
\end{array}{\!\!\!\!\!\!}
\end{equation}
The first line gives
\begin{align*}
\norm{u^{n-1}-u}^{2} + t_{n}^{2}\norm{b^{n-1}-b}^{2} & \geq \norm{u^{n}-u}^{2} + t_{n+1}^{2}\norm{b^{n}-b}^{2} + \norm{u^{n-1}-v^{n}}^{2}\\
& \quad + (t_{n}^{2}-t_{n+1}^{2}) \norm{b^{n}-b}^{2}.
\end{align*}
Summing this inequality from $n=1$ to $n = N$, we get
\begin{align*}
\sum_{n=1}^{N}\norm{u^{n-1}-v^{n}}^{2} & \leq \norm{u^{0}-u}^{2} + t_{1}^{2}\norm{b^{0}-b}^{2} - \left( \norm{u^{N}-u}^{2} + t_{N+1}^{2}\norm{b^{N}-b}^{2}\right)\\
                                           & \quad + \sum_{n=1}^{N}(t_{n+1}^{2} - t_{n}^{2})\norm{b^{n}-b}^{2}.
\end{align*}
Now, since $\norm{b^{n}-b}^{2}$ is bounded and it holds that
\begin{equation*}
	\sum_{n=1}^{\infty}\abs{t_{n}^{2}-t_{n+1}^{2}} = \sum_{n=1}^{\infty}\abs{t_{n}-t_{n+1}}\abs{t_{n}+t_{n+1}} \leq 2\overline{t}\sum_{n=1}^{\infty}\abs{t_{n}-t_{n+1}}<\infty
\end{equation*}
by our assumption, we can  conclude that
\begin{equation*}
	\sum_{n=1}^{\infty}\norm{u^{n-1}-v^{n}}^{2} <\infty,
\end{equation*}
i.e., by~\eqref{eq:newDR-aux3}, we have
\begin{equation*}
	\lim_{n\to\infty} u^{n-1}-v^{n} = \lim_{n\to\infty} a^{n}+b^{n-1} = 0.
\end{equation*}
This expression shows that $v^{n}$ and $a^{n}$ are also bounded.
Due to the boundedness of $\set{(u^{n},b^{n})}$, we conclude the existence of weak convergence subsequences $\set{u_{n_l}}_l$ and $\set{b_{n_l}}_l$ such that
\begin{equation*}
	u^{n_{l}}\weakto u^{*},\quad b^{n_{l}}\weakto b^{*},
\end{equation*}
and by the above limits, we also have
\begin{equation*}
	v^{n_{l}+1}\weakto u^{*},\quad a^{n_{l}+1}\weakto b^{*}.
\end{equation*}
From~\cite[Corollary 3]{bauschke2009note} it follows that $(u^{*},b^{*})\in S(A,B)$. This shows that $\{(u^{n},b^{n})\}$ has a weak cluster point and that all such points are in $S(A,B)$.
Now we deduce from~\eqref{eq:newDR-est2} that
\begin{equation*}
\begin{array}{ll}
  \norm{u^{n}-u^{*}}^{2} + (t^{*})^{2}\norm{b^{n}-b^{*}}^{2} & \leq   \norm{u^{n-1}-u^{*}}^{2} + (t^{*})^{2}\norm{b^{n-1}-b^{*}}^{2} \vspace{1ex}\\
  &+ \abs{t_{n}^{2}-(t^{*})^{2}}\left|\norm{b^{n-1}-b^{*}}^{2} - \norm{b^{n}-b^{*}}^{2}\right|.
\end{array}
\end{equation*}
Since $\norm{b^{n}-b^{*}}^{2}$ is bounded and $t^n\to t^{*}$, this shows that the sequence $\{(u^{n},b^{n})\}$ is quasi-Fejer convergent to the extended solution set $S(A,B)$ with respect to the distance $d( (u,b),(z,w) ) = \norm{u-z}^{2} + (t^{*})^{2}\norm{b-w}^{2}$.
Thus, similar to the proof of~\cite[Theorem 1]{svaiter2011weak}, we conclude that the whole sequence $\set{(u^{n},b^{n})}$ weakly converges  to an element of $S(A,B)$.
\end{proof}

%%%%%%%%%%%%%%%%%%%%%%%%%%%%%%%%%%%%%%%%%%%%%%%
%%% 4. An adaptive step-size for DR methods
%%%%%%%%%%%%%%%%%%%%%%%%%%%%%%%%%%%%%%%%%%%%%%%
\section{An adaptive step-size for DR methods}\label{sec:adaptive-stepsize}
The step-size $t_n$ suggested by \eqref{eq:adaptive_stepsize} or by \eqref{eq:adaptive_stepsize2} is derived from our analysis of a linear case and
it does not guarantee the convergence in general.
In this section, we suggest modifying this step-size so that we can prove the convergence of the DR scheme.
We build our adaptive step-size based on two insights:
\begin{itemize}
\item The estimates of the eigenvalues of the DR-iteration in the linear case from Section~\ref{subsec:adaptive-1} motivated the adaptive stepsize
\begin{equation}\label{eq:dg-adapt-stepsize}
t_n = \frac{\norm{u^n}}{\norm{Bu^n}}
\end{equation}
for single-valued $B$ is single-valued and
for the general case, we consider
\begin{equation}\label{eq:dg-adapt-stepsize2}
t_n = \frac{\norm{J_{t_{n-1}B}y^{n-1}}}{\norm{y^{n-1} - J_{t_{n-1}B}y^{n-1}}}t_{n-1}
\end{equation}
from Section~\ref{subsec:DR_scheme2}. 
\item Theorem~\ref{th:main_theorem} ensures the convergence of the non-stationary DR-iteration as soon as the stepsize sequence is convergent with summable increments.
\end{itemize}
However, the sequences~\eqref{eq:dg-adapt-stepsize} and~\eqref{eq:dg-adapt-stepsize2} are not guaranteed to converge (and numerical experiments indicate that, indeed, divergence may occur).  
Here is a way to adapt the sequence~\eqref{eq:dg-adapt-stepsize} to produce a suitable stepsize sequence in the single-valued case:
\begin{enumerate}
\item Choose safeguards $0 <  t_{\min} < t_{\max}<\infty$, a summable ``conservation sequence'' $\omega_n\in (0,1]$ with $\omega_0 = 1$ and start with $t_{0} = 0$.
\item Let $\proj_{[\gamma, \rho]}(\cdot)$ be the projection onto a box $[\gamma, \rho]$. We construct $\set{t_n}$ as
\begin{equation}\label{eq:dg-adapt-stepsize-3}
t_n =(1-\omega_n)t_{n-1} +  \omega_n \proj_{[t_{\min},t_{\max}]}\left(\frac{\norm{u^n}}{\norm{Bu^n}}\right).
\end{equation}
\end{enumerate}
The following lemma ensures that this will lead to a convergent sequence $\set{t_n}$.

%%% Lemma 4.1.
\begin{lemma}\label{le:convergent_t_n}
Let $\set{\alpha_n}$ be a bounded sequence, i.e., $\underline{\alpha}\leq \alpha_n \leq \bar{\alpha}$, and $\set{\omega_n}\subset (0,1]$ such that $\sum_{n=0}^{\infty} \omega_n < \infty$ and $\omega_0 = 1$.
Then, the sequence $\set{\beta_n}$ defined by $\beta_{0} = 0$ and
\begin{equation*}
\beta_n = (1-\omega_n)\beta_{n-1} + \omega_{n}\alpha_n,
\end{equation*}
is in $[\underline{\alpha}, \bar{\alpha}]$ and converges to some $\beta^{\ast}$ and it holds that $\sum_{n=0}^{\infty}\abs{\beta_{n+1}-\beta_{n}}<\infty$.
\end{lemma}

%% The proof of Lemma 4.1.
\begin{proof}
Obviously, $\beta_0 = \alpha_0$ and since $\beta_n$ is a convex combination of $\alpha_n$ and $\beta_{n-1}$, one can easily see that $\beta_n$ obeys the same bounds as $\alpha_n$, i.e. $\underline{\alpha} \leq \beta_n\leq \bar{\alpha}$.
Moreover, it holds that
\begin{equation*}
\beta_n - \beta_{n-1} = \omega_n \alpha_n + (1-\omega_n)\beta_{n-1} - \beta_{n-1} =   \omega_{n}(\alpha_{n} - \beta_{n-1}),
\end{equation*}
thus   $\abs{\beta_n-\beta_{n-1}}\leq \omega_n(\bar{\alpha} - \underline{\alpha})$ from which the assertion follows, since $\omega_{n}$ is summable.
\end{proof}
%% End of the proof.

Clearly, if we apply Lemma \ref{le:convergent_t_n} to the sequence $\set{t_n}$ defined by \eqref{eq:dg-adapt-stepsize-3}, then it converges to some $t^{\ast}$.
% This is our adaptive stepsize for the case where $B$ is single-valued.
% In practice, a choice $\omega_{n} = 2^{-\frac{(n-1)}{100}}$ works well.

% If $B$ is not single-valued, then the step-size  \eqref{eq:dg-adapt-stepsize-3} is not suitable. 
We use a similar trick to construct an adaptive stepsize based on the choice \eqref{eq:dg-adapt-stepsize2} in the case of multi-valued operators.
More precisely, we construct $\set{t_n}$ as follows:
\begin{enumerate}
\item Choose safeguards $0 <  \kappa_{\min} < \kappa_{\max}<\infty$, a summable ``conservation sequence'' $\set{\omega_n}\subset (0,1]$, and $t_{0} = 1$.
\item We construct $\set{t_n}$ as
\begin{equation}\label{eq:weighted_stepsize}
\begin{array}{ll}
\kappa_n &:= \proj_{[\kappa_{\min}, \kappa_{\max}]}\left( \frac{\norm{J_{t_{n-1}B}y^{n-1}}}{\norm{y^{n-1} - J_{t_{n-1}B}y^{n-1}}}\right),\vspace{1ex}\\
t_{n} &:= \nu_nt_{n-1},~~~\text{where}~~~\nu_n := 1-\omega_{n} + \omega_n\kappa_n.
\end{array}
\end{equation}
\end{enumerate}
In this case we get that $t_{n} = \prod_{k=1}^{n}\nu_{k} t_{0}$ and since
$\abs{\nu_{n}-1} = \omega_{n}\abs{\kappa_{n}-1}$ and $\kappa_{n}$ is bounded, the summability of $\omega_{n}$ implies summability of $\abs{\nu_{n}-1}$. This implies that $\prod_{k=1}^{\infty}\nu_{k}$ converges to some positive value and and thus, $t_{n}\to t^{*}>0$, too.
% In our implementation below, we again choose $\omega_n = 2^{-\frac{n}{100}} \in (0, 1)$ and it works well.

The stepsize sequence $\set{t_n}$ constructed by either \eqref{eq:dg-adapt-stepsize-3} or \eqref{eq:weighted_stepsize} fulfills the conditions of Theorem~\ref{th:main_theorem}.
Hence, the convergence of the nonstationary DR scheme using this adaptive stepsize follows as a direct consequence.
We will provide guidelines on how to choose the safeguards and the conservation sequence in practice in Section~\ref{sec:experiments-dr}.

%%% 5. Application to ADMM
\section{Application to ADMM}\label{sec:admm}
It is well-known that the alternating direction method of multipliers (ADMM) for convex optimization with linear constraint can be interpreted as the DR method on its dual problem, see, e.g. \cite{eckstein1992douglas}.
In this section, we apply our adaptive stepsize to ADMM to obtain a new variant for solving the following constrained problem:
\begin{equation}\label{eq:constr_cvx}
\min_{u, v}\Big\{ \phi(u, v) = \varphi(u) + \psi(v) ~\mid~ Du + Ev = c \Big\},
\end{equation}
where $\varphi : \Hbs_u \to\Rext$, $\psi : \Hbs_v\to\Rext$ are two proper, closed, and convex functions, $D : \Hbs_u\to\Hbs$ and  $E : \Hbs_v \to\Hbs$ are two given bounded linear operators, and $c\in\Hbs$.

%%% Monotone inclusion of the dual.
The dual problem associated with \eqref{eq:constr_cvx} becomes 
\begin{equation}\label{eq:dual_cvx}
\min_x\Big\{ \varphi^{\ast}(D^Tx) + \psi^{\ast}(E^Tx) - c^Tx \Big\},
\end{equation}
where $\varphi^{\ast}$ and $\psi^{\ast}$ are the Fenchel conjugate of $\varphi$ and $\psi$, respectively.
The optimality condition of \eqref{eq:dual_cvx} becomes 
\begin{equation}\label{eq:opt_dual_cvx}
0 \in \underbrace{D\partial{\varphi}^{\ast}(D^Tx) - c}_{Ax} + \underbrace{E\partial{\psi}^{\ast}(E^Tx)}_{Bx},
\end{equation}
which is of the form \eqref{eq:mono_inc}.

In the stationary case, ADMM is equivalent to the DR method applying to the dual problem \eqref{eq:opt_dual_cvx}, see, e.g., \cite{eckstein1992douglas}.
However, for the non-stationary DR method, we can derive a different parameter update rule for ADMM.
Let us summarize this result into the following theorem for the non-stationary scheme  \eqref{eq:new_DR_scheme}.
The proof of this theorem is given in Appendix \ref{apdx:th:admm}.

\begin{theorem}\label{th:admm}
Given $0 < t_{\min} < t_{\max} <+\infty$, the ADMM scheme for solving \eqref{eq:constr_cvx} derived from the non-stationary DR method~\eqref{eq:new_DR_scheme} applying to \eqref{eq:opt_dual_cvx} becomes:
\begin{equation}\label{eq:DR_admm2}
\left\{\begin{array}{ll}
u^{n+1} &:= \displaystyle\argmin_u \Big\{ \varphi(u) - \iprods{Du, w^n} + \frac{t_{n-1}}{2} \Vert Du + Ev^n - c\Vert^2 \Big\},\\
v^{n+1} &:= \displaystyle\argmin_v \Big\{ \psi(v) - \iprods{Ev, w^n} + \frac{t_{n-1}}{2}\Vert Du^{n+1} + Ev - c\Vert^2 \Big\},\\
w^{n+1} &:= w^n -  t_{n-1}(Du^{n+1} + Ev^{n+1} - c), \\
t_{n}      &:= (1-\omega_n)t_{n-1} + \omega_n\proj_{[t_{\min}, t_{\max}]}\left(\frac{\norm{w^{n+1}}}{\norm{Ev^{n+1}}}\right),~~\omega_n\in (0, 1).
\end{array}\right.
\end{equation}
Consequently, the sequence $\set{w^n}$ generated by \eqref{eq:DR_admm2} weakly converges to a solution $x^{\ast}$ of the dual problem \eqref{eq:dual_cvx}.
\end{theorem}

The ADMM variant \eqref{eq:DR_admm2} is essentially the same as the standard ADMM, but its parameter $t_n$ is  adaptively updated. 
This rule is different from \cite{he2000alternating,xu2016adaptive}. 

\section{Numerical experiments}\label{sec:numerical-experiments}
In this section we provide several numerical experiments to
illustrate the influence of the stepsize and the adaptive choice in
practical applications.
Although we motivate the adaptive stepsize only for linear problems, we will apply it to problems that do not fulfill this assumption since
the convergence of the method is ensured by Theorem~\ref{th:main_theorem} in all cases.
We also note that the steps of the non-stationary method may be more costly than the one with constant stepsize, if the evaluation of the resolvents is costly and the constant stepsize can be leveraged to precompute something.
This is the case when $A$ and/or $B$ is linear and the resolvents involve the solution of a linear system for which a matrix factorization can be precomputed.
However, there are tricks to overcome this issue, see~\cite[pages 28-29]{boyd2011distributed}, but we will not go in more detail here.

For the Douglas-Rachford method we just provide illustrative examples since we are not aware of any adaptive rule that applies to the Douglas-Rachford method in the general case of monotone operators. For the ADMM there are several other adaptive rules available and we do a comparison in Section~\ref{sec:experiments-admm}.

\subsection{Experiments for non-stationary Douglas-Rachford}
\label{sec:experiments-dr}
We provide four numerical examples to illustrate the new adaptive DR scheme \eqref{eq:new_DR_scheme} on some well-studied problems in the literature.
The stepsizes~\eqref{eq:dg-adapt-stepsize-3} (in the case of single valued $B$) and \eqref{eq:weighted_stepsize} (in the case of multivalued $B$) come with new parameters: the safeguards $t_{\min/\max}$ and $\kappa_{\min/\max}$ and a ``conservation'' term $\omega_{n}$. Since $B$ is single valued is all experiments we always used~\eqref{eq:dg-adapt-stepsize-3} and we also fixed $t_{\min}=10^{-4}$, $t_{\max} = 10^{4}$ and $\omega_{n} = 2^{-n/100}$ for all experiments.

%%% 6.1.1 The linear toy example
\subsubsection{The linear toy example}%\label{sec:linear-examples}
We start with the linear toy example from Section~\ref{sec:motivating-example}. The residual sequence along the iterations is shown on the left of Figure~\ref{fig:motivating-example-adaptive-residual-stepsize}. 
\begin{figure}[htp!]
\begin{center}
\includegraphics[width=0.57\textwidth]{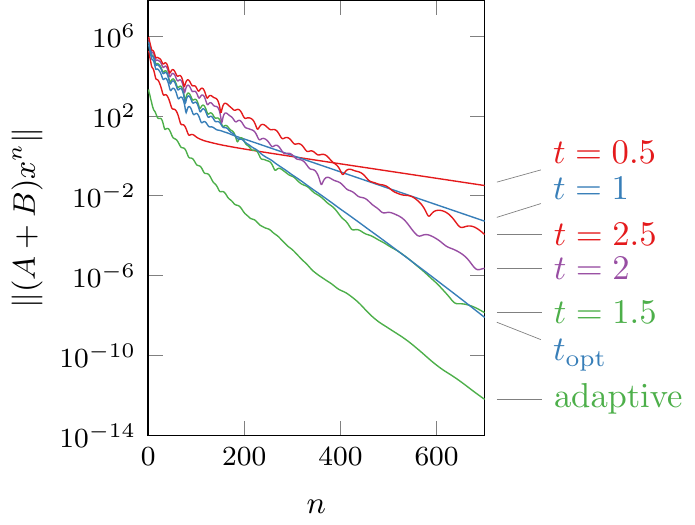}~%
\includegraphics[width=0.42\textwidth]{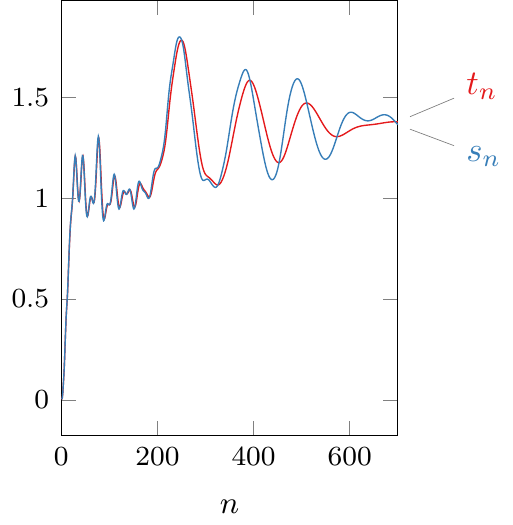}
\vspace{-3ex}
\caption{Results for the linear problem from Section~\ref{sec:motivating-example} using fixed and the adaptive stepsizes. 
Left: Residual sequences. Right: Auxiliary sequence $s_{n} = \norm{u^{n}}/\norm{Bu^{n}}$ and the stepsize $t_n$.}\label{fig:motivating-example-adaptive-residual-stepsize}
\vspace{-2ex}
\end{center}
\end{figure}
Additionally, we determined the stepsize $t_{\text{opt}}$ that leads to the smallest asymptotic convergence rate, i.e. to the smallest spectral radius of the iteration map $H_{t_{\text{opt}}}$  (in this case $t_{\text{opt}} = 1.367$) and also plot the corresponding residual sequence with this optimal constant stepsize in the same figure.
The adaptive stepsize does indeed improve the convergence considerably both by using small steps in the beginning and automatically tuning to a stepsize $t$ that is close to the optimal one
(cf.~Figure~\ref{fig:motivating-example-residual}, right). 
It also outperforms the optimal constant stepsize $t_{\mathrm{opt}}$.

%%% 6.1.2. LASSO problems.
\subsubsection{LASSO problems}%\label{sec:lasso-problems}
The LASSO problem is the minimization problem
\begin{equation}\label{eq:lasso}
\min_{x} \Big[ F(x) = \tfrac12\norm{Kx-b}_{2}^{2} + \alpha\norm{x}_{1} \Big]
\end{equation}
and is also known as basis pursuit denoising~\cite{tibshirani1996regression}.
We will treat this with the Douglas-Rachford method as follows: 
We set $F = f + g$ with 
\begin{align*}
g(x)& = \tfrac12\norm{Kx-b}_{2}^{2},& B&= \nabla g(x) = K^{T}(Kx-b)\\
f(x)& = \alpha\norm{x}_{1},& A &= \partial f(x).
\end{align*}
In this particular example we take $K\in\R^{100\times 1000}$ with orthonormal rows,
and hence, by the matrix inversion lemma, we get
\begin{equation*}
(I+tB)^{-1}x = (I+tK^{T}K)^{-1}(x + tK^{T}b) = (I - \tfrac{t}{t+1}K^{T}K)(x+tK^{T}B).
\end{equation*}
The resolvent of $A$ is the so-called soft-thresholding operator:
\begin{equation*}
(I+tA)^{-1}x = \max(\abs{x}-t\alpha,0)\sign(x).
\end{equation*}
Note that $B$ is single-valued and $A$ is a subgradient and hence, the adaptive stepsize $t_n$ computed by \eqref{eq:dg-adapt-stepsize-3} does apply.
Figure~\ref{fig:lasso} shows the result of the Douglas-Rachford iteration with constant and adaptive stepsizes, and also a comparison with the FISTA~\cite{beck2009fast} method. 
(Note that if $K$ would not have orthonormal rows, one would have to solve a linear system at each Douglas-Rachford step which would make the comparison with FISTA by iteration count unfair.)
As shown in this plot, the adaptive stepsize again automatically tunes to a stepsize close to $10$ which, experimentally, seems to be the best constant stepsize for this particular instance.

\begin{figure}[htp!]
\centering
\includegraphics[width=0.7\textwidth]{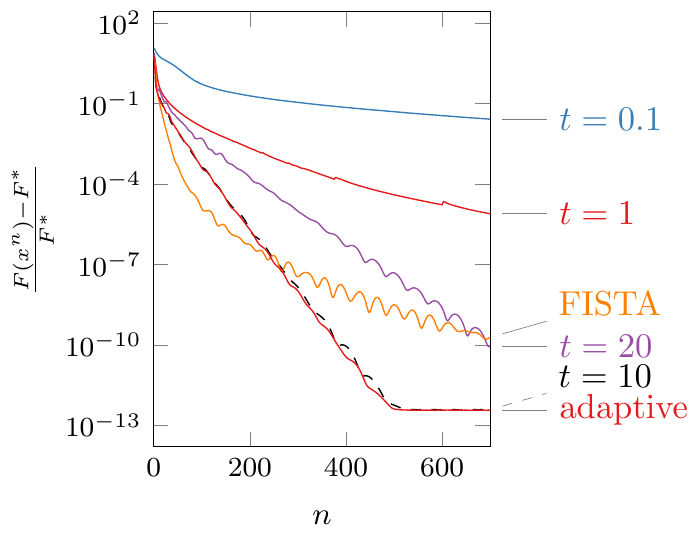}
\vspace{-1ex}
\caption{The convergence behavior of  the Douglas-Rachford iteration and  the FISTA method on a LASSO problem using fixed and the adaptive stepsizes.}\label{fig:lasso}
\vspace{-1ex}
\end{figure}

%% 6.1.3. Convex-concave saddle-point problems
\subsubsection{Convex-concave saddle-point problems}
%\label{sec:conv-conv-saddle}
Let $X$ and $Y$ be two finite dimensional Hilbert spaces, $K:X\to Y$ be a bounded linear operator and $f:X\to\Rext$ and $g:Y\to\Rext$ be two proper, convex and lower-semicontinuous functionals. The saddle point problem then reads as
\begin{equation*}
\min_{x\in X}\max_{y\in Y}\Big\{ f(x) + \iprod{Kx,y} - g(y) \Big\}.
\end{equation*}
Saddle points $(x^{*},y^{*})$ are characterized by the inclusion
\begin{equation*}
0\in \begin{bmatrix}
  \partial f & K^{T}\\ -K  & \partial g
\end{bmatrix}
\begin{bmatrix} x^{*}\\y^{*}\end{bmatrix}.
\end{equation*}
To apply the Douglas-Rachford method we split the optimality system as follows. We denote $z=(x,y)$ and set
\begin{equation*}
A =\begin{bmatrix}
  \partial f & 0\\   0 & \partial g
\end{bmatrix},~~~~~
B = \begin{bmatrix}
  0 & K^{T}\\  -K & 0
\end{bmatrix},
\end{equation*}
(cf.~\cite{oconnor2014primal,bredies2015dr}). 
The operator $A$ is maximally monotone as a subgradient and $B$ is linear and skew-symmetric, hence maximally monotone and even continuous.

One standard problem in this class in the so-called
Rudin-Osher-Fatemi model for image denoising~\cite{rudin1992nonlinear}, also known as
total variation denoising. For a given noisy image $u_{0}\in\R^{M\times N}$ one seeks a
denoising image $u$ as the minimizer of
\begin{equation*}
\min_{u}\Big\{ \tfrac12\norm{u-u_{0}}_{2}^{2}+\lambda \norm{\abs{\nabla u}}_{1} \Big\},
\end{equation*}
where $\nabla u\in\R^{M\times N\times 2}$ denotes the discrete gradient of $u$ and $\abs{\nabla u}$ denotes the components-wise magnitude of this gradient. 
The penalty term $\norm{\abs{\nabla u}}_{1}$ is the discretized total variation, and $\lambda > 0$ is a regularization parameter.
The saddle point form of this minimization problem is
\begin{equation*}
\min_{u}\max_{\abs{\phi}\leq \lambda}\set{ \tfrac12\norm{u-u_{0}}_{2}^{2} + \iprod{\nabla u,\phi}}.
\end{equation*}
We test our DR scheme \eqref{eq:new_DR_scheme} using the adaptive stepsize $t_n$ and compare with two constant stepsizes $t=1$ and $t=13$.
The constant stepsize $t=13$ seems to be the best among many trial stepsizes after tuning. 
The convergence behavior of these cases is plotted in Figure \ref{fig:saddle-rof} for one particular image called \texttt{auge} of the size $256\times 256$.
As we can see from this figure that the adaptive stepsize has a good performance and is comparable with the best constant stepsize in this example ($t=13$).
\begin{figure}[htb]
  \centering
  \includegraphics[width=0.5\textwidth]{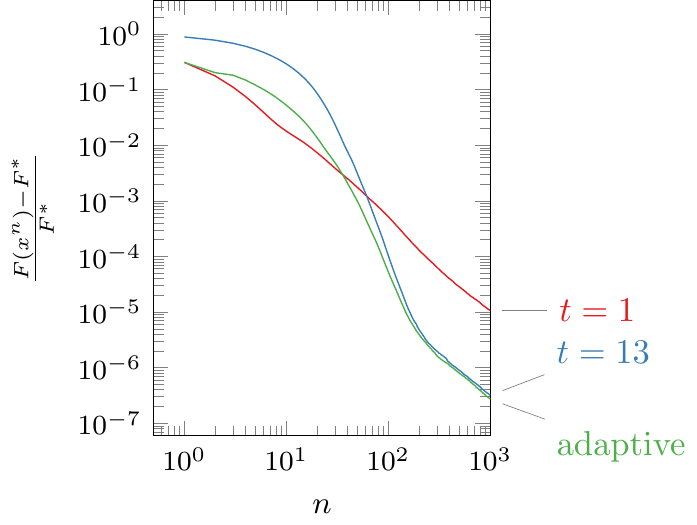}
\vspace{-2ex}    
  \caption{The decrease of the objective values of three DR variants in the total variation denoising problem.}\label{fig:saddle-rof}
\vspace{-2ex}  
\end{figure}

\subsection{Experiments for ADMM with an adaptive stepsize}
\label{sec:experiments-admm}

In this subsection we verify the performance of the our adaptive ADMM variant  \eqref{eq:DR_admm2}. We follow the comparison from~\cite{xu2017adaptive} where several adaptive variants of ADMM are compared. However, we only compare the methods of ADMM that do not involve relaxation, since we did not consider relaxation in this paper.

In our comparison we compare the ADMM with constant stepsize which is fixed ad-hoc, the adaptive rule of He~\cite{he2000alternating} which is based on residual balancing (RB), the adaptive ADMM (AADMM) from~\cite{xu2016adaptive} and our approach from Theorem~\ref{th:admm}. We used five different test problems from the comparison in~\cite{xu2017adaptive}: Elastic net regression, LASSO regression, quadratic programming, consensus $\ell^{1}$-regularized logistic regression, and SVM for classification (see~\cite[Section 6]{xu2017adaptive} for details).
We also use the code released online from~\cite{xu2016adaptive}.

Table~\ref{tab:results_admm} summarizes the results for average number of iterations for 50 runs on random instances of the same size. Note that both the RB ADMM and the AADMM do guarantee convergence only if the the adaptivity is switched of at a certain point while our rule comes with a convergence guarantee.\footnote{The paper~\cite{xu2017adaptive} has a convergence guarantee for an adaptive relaxed method, but this does not apply to the methods used in this comparison and is not included since it also involves relaxation.} Table~\ref{tab:results_admm} shows that our adaptive method consistently performs good.

\begin{figure}[hpt!] %[htb]
  \centering
  \subfigure[Elastic net]{
    \includegraphics[width=5cm]{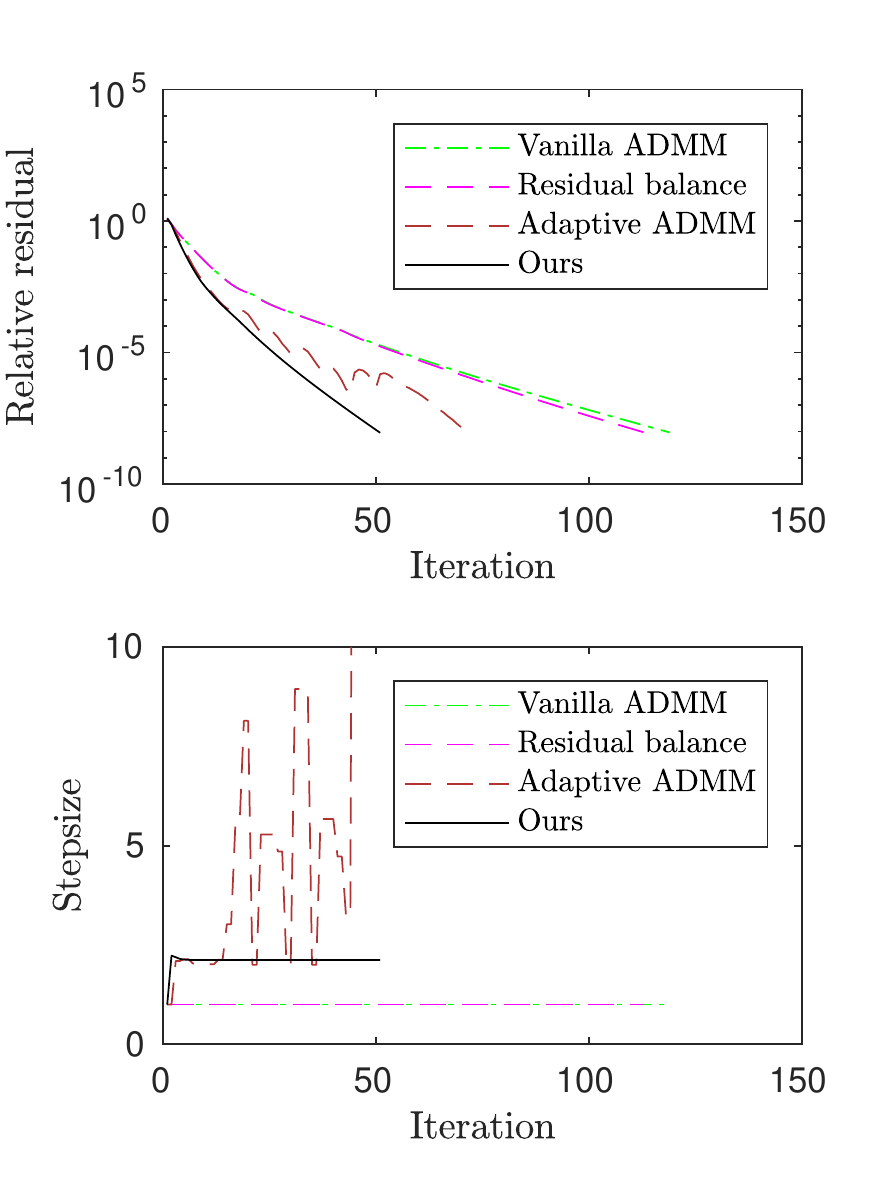}}
  \subfigure[LASSO]{
    \includegraphics[width=5cm]{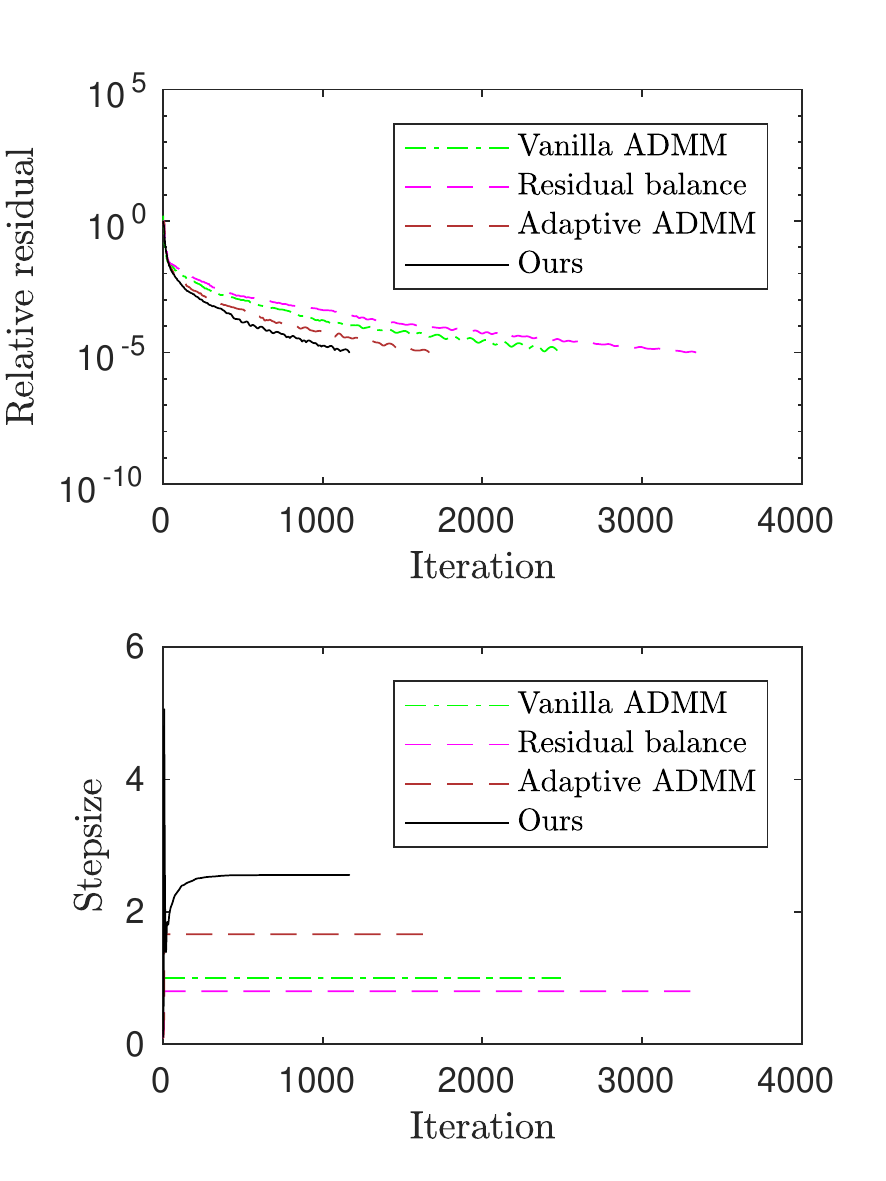}}
  \subfigure[Quadratic programming]{
    \includegraphics[width=5cm]{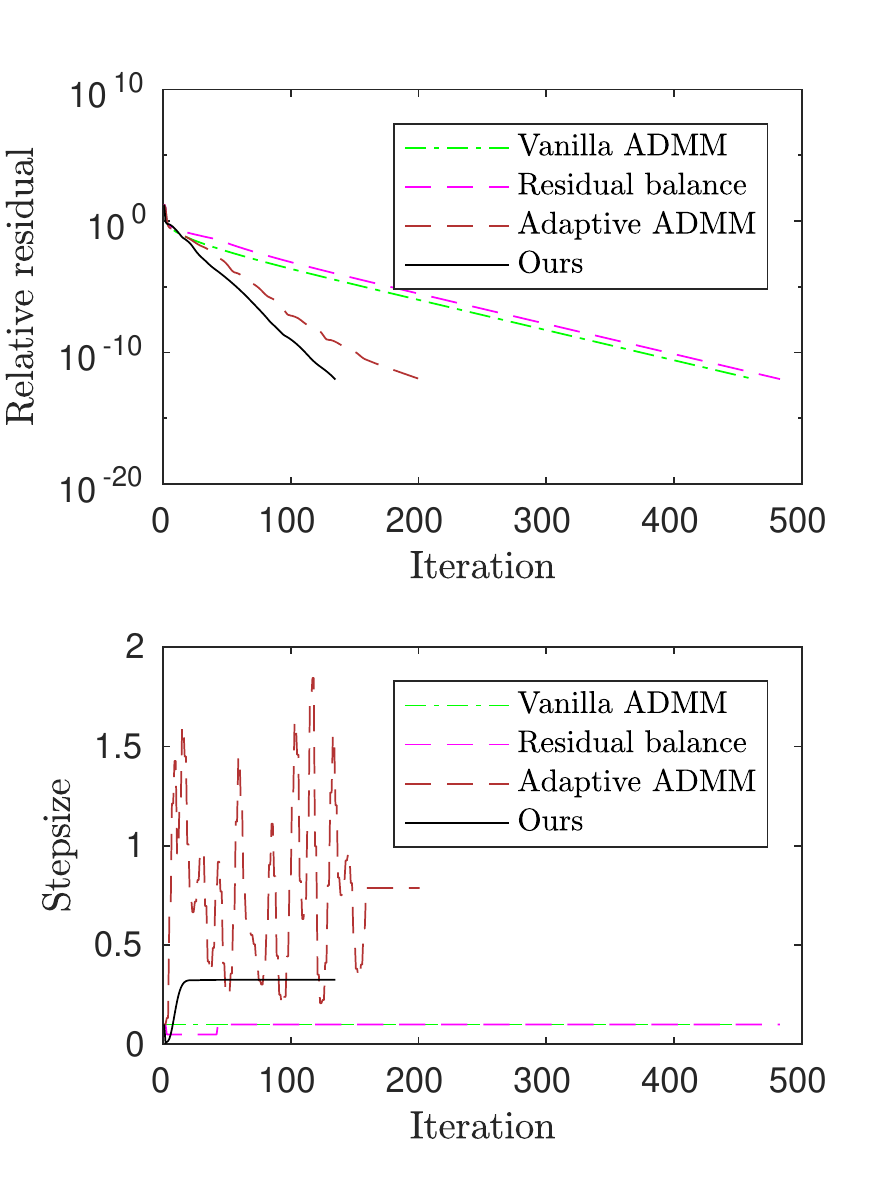}}
  \subfigure[Logistic regression]{
    \includegraphics[width=5cm]{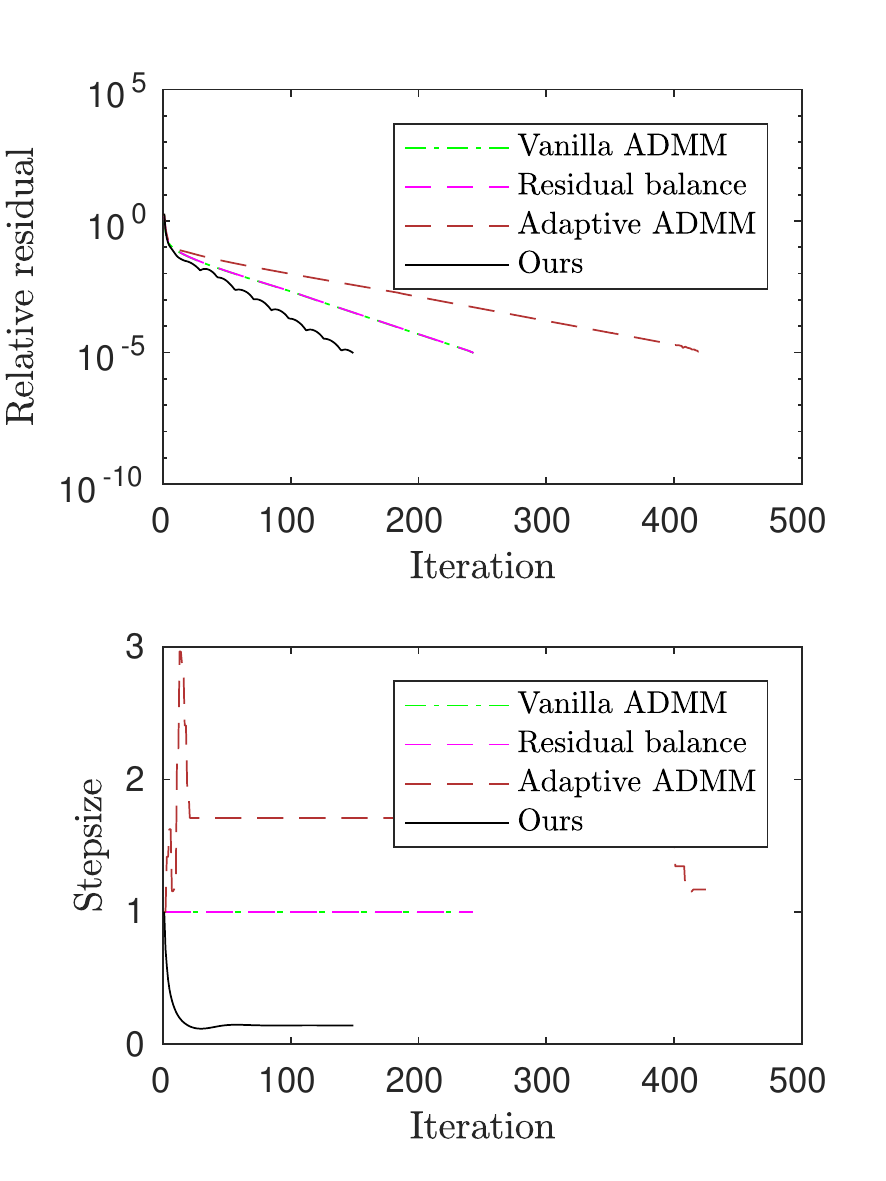}}
  \caption{Example runs of the different adaptive methods. Upper plots show the relative residual (cf.~\cite[Section 3.3.1]{boyd2011distributed} or~\cite[Section 4.3]{xu2016adaptive}), lower plots show the stepsizes, respectively.}
  \label{fig:admm_example_runs}
\end{figure}

%%%%%
% New table for comparison with Xu et al.
\begin{table}[tb]
  \centering%\small
\begin{tabular}{llr@{$\pm$}lr@{$\pm$}lr@{$\pm$}lr@{$\pm$}l}
  \toprule
  \input{tables/comparison_ADMM.tex}
  \bottomrule
\end{tabular} 
\caption{Results for the comparison of different apadtive stepsizes for ADMM. We compare the number of iterations needed for the methods to reach a given tolerance as in \cite{xu2016adaptive}. We report mean($\pm$ standard deviation) for 50 runs on random instances.}
\label{tab:results_admm}
\end{table}
%%%%%

Figure~\ref{fig:admm_example_runs} shows example runs for four of the five problem (the fifth being the SVM classification and is omitted due to space reasons). One observes that residual balancing often fails to make progress towards a favorable stepsize and that AADMM sometimes shows large oscillations in the stepsizes. Our method leads to a stepsize sequence that stabilizes quickly and leads to good reduction of the residual.

\section{Conclusion}\label{sec:conclusion}
%% General idea: 
% (a) address key question;  (b) summarize the key contribution; (c) open questions
We have attempted to address one fundamental practical issue in the well-known DR method: step-size selection.
This issue has been standing for a long time and has not adequately been well-understood.
In this paper, we have proposed an adaptive step-size that is derived from an observation of the linear case.
Our non-stationary DR method is new; it is derived from the iteration for single-valued $B$ and differs from the standard non-stationary iteration considered previously, e.g. in~\cite{Liang2017}.
Our stepsize remains heuristic in the general case, but we can guarantee a global convergence of the DR method.
% These two aspects, a new adaptive stepsize and convergence of a non-stationary DR scheme, are our main contribution.
As a byproduct, we have also derived a new ADMM variant that uses a simple adaptive stepsize and has a convergence guarantee.
This is practically significant since ADMM has been widely used in many areas in the last two decades.
Our finding also opens some future research ideas:
Although we gained some insight, the linear case is still not properly understood.
Since our heuristic applies to general $A$ and $B$, there is the possibility to investigate, which operators should be used as ``$B$'' to compute the adaptive stepsize.
% Moreover, the connection of our heuristic to other criteria such as He's stepsize or the method from~\cite{ghadimi2015optimal} could be analyzed.
As shown in~\cite{oconnor2014primal}, one can rescale convex-concave saddle point problems to use two different stepsizes for the Douglas-Rachford method, and one may extend our heuristic to this case.
Moreover, the convergence speed of the non-stationary method under additional assumptions such as Lipschitz continuity or coercivity could be analyzed.
Finally, an adaptive rule for the relaxed DR method would be of interest.

\subsection*{Acknowledgement}

We thank Zheng Xu and Tom Goldstein for sharing the ADMM code.

\appendix

\section{The proof of Theorem \ref{th:admm}}\label{apdx:th:admm}
Let us assume that we apply \eqref{eq:new_DR_scheme} to solve the optimality condition \eqref{eq:opt_dual_cvx} of the dual problem \eqref{eq:dual_cvx}.
From \eqref{eq:new_DR_scheme}, i.e., 
\begin{equation*}
y^{n+1} = J_{t_{n}A}( (1+\kappa_n)J_{t_{n-1}B}y^{n} - \kappa_n y^{n}) + \kappa_n\left( y^{n} -J_{t_{n-1}B}y^{n}\right),
\end{equation*}
we define $w^{n+1} := J_{t_{n-1}B}y^{n}$ and $z^{n+1} := J_{t_{n}A}( (1 + \kappa_n)w^{n+1} - \kappa_ny^n)$ to obtain
\begin{equation*}
\begin{cases}
w^{n+1} &:= J_{t_{n-1}B}y^{n}\\ 
z^{n+1} &:= J_{t_{n}A}((1 + \kappa_n)w^{n+1} - \kappa_n y^n) \\ 
y^{n+1} &= z^{n+1} + \kappa_n(y^n- w^{n+1}).
\end{cases}
\end{equation*}
Shifting up this scheme by one index and changing the order, we obtain
\begin{equation*}
\begin{cases}
z^{n} &= J_{t_{n-1}A}((1 + \kappa_{n-1})w^{n} - \kappa_{n-1}y^{n-1}) \\ 
y^{n} &= z^{n} + \kappa_{n-1}(y^{n-1} - w^{n}) \\ 
w^{n+1} &= J_{t_{n-1}B}y^n = J_{t_{n-1}B}\left(z^n + \kappa_{n-1}(y^{n-1} - w^n)\right).
\end{cases}
\end{equation*}
Let $(1 + \kappa_{n-1})w^{n} - \kappa_{n-1}y^{n-1} = x^n + w^n$. 
This gives $x^n = \kappa_{n-1}(w^n - y^{n-1})$
and hence, $z^n + \kappa_{n-1}(y^{n-1} - w^n) = z^n - x^n$ and $x^{n+1} = \kappa_n(w^{n+1} - y^{n}) =  \kappa_n(w^{n+1} - z^n + x^n)$.
Substituting these into the above expression of the DR scheme, we obtain 
\begin{equation}\label{eq:DR_iteration4}
\left\{\begin{array}{ll}
z^{n} &= J_{t_{n-1}A}(x^{n} + w^n)\\
w^{n+1} &= J_{t_{n-1}B}(z^{n} - x^n)\\
x^{n+1} &= \kappa_n(x^n + w^{n+1} - z^{n}),
\end{array}\right.
\end{equation}
where $x^n = \kappa_{n-1}(w^{n} - y^{n-1})$.

From $z^{n} = J_{t_{n-1}A}(w^{n} + x^n)$, we have $z^n = (I + t_{n-1}A)^{-1}(w^n + x^n)$ or 
\begin{equation*}
0 \in z^n - w^n - x^n + t_{n-1}(D\nabla{\varphi^{*}}(D^Tz^n) -  c).
\end{equation*}
Let $u^{n+1} \in \nabla{\varphi^{\ast}}(D^Tz^n)$, which implies $D^Tz^n \in \partial{\varphi}(u^{n+1})$. 
Hence, we have $z^n - w^n - x^n + t_{n-1}(Du^{n+1} - c) = 0$, therefore  $D^Tz^n = D^T(w^n + x^n - t_{n-1}(Du^{n+1} - c)) \in \partial{\varphi}(u^{n+1})$. This condition leads to 
\begin{equation*}
0 \in D^T(t_{n-1}(Du^{n+1} - c) - x^n - w^n) + \partial{\varphi}(u^{n+1}).
\end{equation*}
This is the optimality condition of 
\begin{equation*}
u^{n+1} = \argmin_u\set{ \varphi(u) + \frac{t_{n-1}}{2}\Vert Du - c - t_{n-1}^{-1}(x^n + w^n)\Vert^2}.
\end{equation*}
%%%%
Similarly, from $w^{n+1} = J_{t_{n-1}B}(z^n - x^n)$, if we define $v^{n+1} \in \nabla{\psi^{\ast}}(E^Tw^{n+1})$, 
then we can also derive that
\begin{equation*}
v^{n+1} = \argmin_v\set{ \psi(v) + \frac{t_{n-1}}{2}\Vert Ev + t_{n-1}^{-1}(x^n - z^n)\Vert^2}.
\end{equation*}
From the line $z^n - w^n - x^n + t_{n-1}(Du^{n+1} - c) = 0$ above, we can write $x^n - z^n = t_{n-1}(Du^{n+1} - c) - w^n$. 
Substituting this expression into the above step, we obtain
\begin{equation*}
\begin{array}{ll}
v^{n+1} 
&= \displaystyle\argmin_v\set{ \psi(v) - \iprods{w^n, Ev} + \tfrac{t_{n-1}}{2}\Vert Ev + Du^{n+1} - c\Vert^2}.
\end{array}
\end{equation*}
This is the second line of \eqref{eq:DR_admm2}.

Next, from $w^{n+1} - z^n + x^n + t_{n-1}Ev^{n+1} = 0$, we have $w^{n} = z^{n-1} - x^{n-1} - t_{n-2}Ev^{n}$. 
This implies $Ev^n = -t_{n-2}^{-1}(x^{n-1} + w^n - z^{n-1})$.
From the last line of \eqref{eq:DR_iteration4}, we have $x^n = \kappa_{n-1}(x^{n-1} + w^n - z^{n-1})$.
Combine these two lines, we get $Ev^n = -\tfrac{1}{\kappa_{n-1}t_{n-2}}x^n = -\frac{1}{t_{n-1}}x^n$ due to the update rule \eqref{eq:adaptive_stepsize2}: $t_{n-1} = \kappa_{n-1}t_{n-2}$.
Substituting $Ev^n  = -\frac{1}{t_{n-1}}x^n$ into the $u$-subproblem, we obtain
\begin{equation*}
\begin{array}{ll}
u^{n+1} %&= \argmin_u\set{ \varphi(u) + \frac{t_{n-1}}{2}\Vert Du - c - t_{n-1}^{-1}w^n + Ev^n\Vert^2}\\
&= \argmin_u\set{ \varphi(u) - \iprods{w^n, Du} + \frac{t_{n-1}}{2}\Vert Du + Ev^n - c\Vert^2}.
\end{array}
\end{equation*}
This is the first line of \eqref{eq:DR_admm2}.

Now, since $z^n = w^n - t_{n-1}(Du^{n+1} - c) + x^n$, and $w^{n+1} = z^n - x^n - t_{n-1}Ev^{n+1}$, combining these expressions, we obtain $w^{n+1} = w^n - t_{n-1}(Du^{n+1} + Ev^{n+1} - c)$.
This is the last line of \eqref{eq:DR_admm2}.

Finally, we derive the update rule for $t_n$.
Indeed, note that $y^n = z^n - x^n$, and $z^n - w^n - x^n + t_{n-1}(Du^{n+1} - c) = 0$. 
These relations show that $y^n = w^n - t_{n-1}(Du^{n+1} - c)$.
Moreover, we also have $w^{n+1} = J_{t_{n-1}B}(z^n - x^n) = J_{t_{n-1}B}(y^n)$.
In this case, we have $J_{t_{n-1}B}(y^n) - y^n = w^{n+1} - w^n + t_{n-1}(Du^{n+1} - c) = -t_{n-1}(Du^{n+1} + Ev^{n+1} - c) + t_{n-1}(Du^{n+1} - c) = -t_{n-1}Ev^{n+1}$.
Hence, we can compute $\kappa_n$ as
\begin{equation*}
\kappa_n := \frac{\norm{J_{t_{n-1}B}(y^n)}}{\norm{y^n - J_{t_{n-1}B}(y^n)}} = \frac{\norm{w^{n+1}}}{t_{n-1}\norm{Ev^{n+1}}}.
\end{equation*}
Using the fact that $t_n := \kappa_nt_{n-1}$, we show that $t_n := \frac{\norm{w^{n+1}}}{\norm{Ev^{n+1}}}$, which is the last line of \eqref{eq:DR_admm2} after projecting and weighting as in Section \ref{sec:adaptive-stepsize}.
Since $\set{w^n}$ is equivalent to the sequence $\set{u^n}$ in the DR scheme \eqref{eq:DR-iter-1} (or equivalently, \eqref{eq:new_DR_scheme}) applying to the dual optimality condition \eqref{eq:opt_dual_cvx} of the dual problem \eqref{eq:dual_cvx}, the last conclusion is a direct consequence of Theorem~\ref{th:main_theorem}.
\Eproof

\bibliographystyle{plain}
\bibliography{ref}

\end{document}

%% file: tables/comparison_ADMM.tex
 &\hspace*{6pt}& \multicolumn{2}{l}{Vanilla ADMM} & \multicolumn{2}{l}{RB ADMM} & \multicolumn{2}{l}{AADMM} & \multicolumn{2}{l}{Ours} \\\toprule
Elastic Net &&  $1198$&$ 145$ & $156 $&$ 18$ & $77 $&$ 15$  & $54 $&$  7$\\
LASSO &&   $1325$&$ 136$ & $1025 $&$ 319$ & $1351 $&$ 826$  & $650 $&$ 75$\\
QP &&   $420$&$ 49$ & $436 $&$ 44$ & $210 $&$ 33$  & $144 $&$ 14$\\
logreg &&   $273$&$ 85$ & $264 $&$ 93$ & $506 $&$ 358$  & $127 $&$ 35$\\
SVM &&  $1690$&$ 329$ & $2189 $&$ 1342$ & $1678 $&$ 1508$  & $878 $&$ 352$\\

%% file: paper_dr_convergence_amo.bbl
\begin{thebibliography}{10}

\bibitem{bauschke2009note}
Heinz~H. Bauschke.
\newblock A note on the paper by {E}ckstein and {S}vaiter on “{G}eneral
  projective splitting methods for sums of maximal monotone operators”.
\newblock {\em SIAM Journal on Control and Optimization}, 48(4):2513--2515,
  2009.

\bibitem{bauschke2017convex}
Heinz~H. Bauschke and Patrick~L. Combettes.
\newblock {\em Convex analysis and monotone operator theory in Hilbert spaces}.
\newblock Springer, second edition edition, 2017.

\bibitem{Bauschke2012}
Heinz~H. Bauschke, Sarah~M. Moffat, and Xianfu Wang.
\newblock Firmly nonexpansive mappings and maximally monotone operators:
  {C}orrespondence and duality.
\newblock {\em Set-Valued and Variational Analysis}, 20(1):131--153, 2012.

\bibitem{beck2009fast}
Amir Beck and Marc Teboulle.
\newblock A fast iterative shrinkage-thresholding algorithm for linear inverse
  problems.
\newblock {\em SIAM Journal on Imaging Sciences}, 2(1):183--202, 2009.

\bibitem{Becker2011a}
S.~Becker, E.~J. Cand\`{e}s, and M.~Grant.
\newblock Templates for convex cone problems with applications to sparse signal
  recovery.
\newblock {\em Math. Program. Compt.}, 3(3):165--218, 2011.

\bibitem{becker2013algorithm}
Stephen Becker and Patrick~L. Combettes.
\newblock An algorithm for splitting parallel sums of linearly composed
  monotone operators, with applications to signal recovery, 2013.
\newblock Arxiv preprint:1305.5828.

\bibitem{boyd2011distributed}
Stephen Boyd, Neal Parikh, Eric Chu, Borja Peleato, and Jonathan Eckstein.
\newblock Distributed optimization and statistical learning via the alternating
  direction method of multipliers.
\newblock {\em Foundations and Trends{\textregistered} in Machine Learning},
  3(1):1--122, 2011.

\bibitem{bredies2015preconditioned}
Kristian Bredies and Hong~Peng Sun.
\newblock Preconditioned {D}ouglas--{R}achford algorithms for {TV}-and
  {TGV}-regularized variational imaging problems.
\newblock {\em Journal of Mathematical Imaging and Vision}, 52(3):317--344,
  2015.

\bibitem{bredies2015dr}
Kristian Bredies and Hongpeng Sun.
\newblock Preconditioned {D}ouglas-{R}achford splitting methods for
  convex-concave saddle-point problems.
\newblock {\em SIAM Journal on Numerical Analysis}, 53(1):421--444, 2015.

\bibitem{bredies2016accelerated}
Kristian Bredies and Hongpeng Sun.
\newblock Accelerated {D}ouglas-{R}achford methods for the solution of
  convex-concave saddle-point problems.
\newblock {\em arXiv preprint arXiv:1604.06282}, 2016.

\bibitem{combettes2004solving}
Patrick~L. Combettes.
\newblock Solving monotone inclusions via compositions of nonexpansive averaged
  operators.
\newblock {\em Optimization}, 53(5-6):475--504, 2004.

\bibitem{combettes2007douglas}
Patrick~L. Combettes and Jean-Christophe Pesquet.
\newblock A {D}ouglas--{R}achford splitting approach to nonsmooth convex
  variational signal recovery.
\newblock {\em IEEE Journal of Selected Topics in Signal Processing},
  1(4):564--574, 2007.

\bibitem{Davis2014}
D.~Davis.
\newblock Convergence rate analysis of the forward-{D}ouglas-{R}achford
  splitting scheme.
\newblock {\em SIAM Journal on Optimization}, 25(3):1760--1786, 2015.

\bibitem{davis2016convergence}
Damek Davis and Wotao Yin.
\newblock Convergence rate analysis of several splitting schemes.
\newblock In {\em Splitting Methods in Communication, Imaging, Science, and
  Engineering}, pages 115--163. Springer International Publishing, 2016.

\bibitem{douglas1956numerical}
Jim~Jr. Douglas and Henry H.~Jr. Rachford.
\newblock On the numerical solution of heat conduction problems in two and
  three space variables.
\newblock {\em Transactions of the American mathematical Society},
  82(2):421--439, 1956.

\bibitem{eckstein1992douglas}
Jonathan Eckstein and Dimitri~P. Bertsekas.
\newblock On the {D}ouglas-{R}achford splitting method and the proximal point
  algorithm for maximal monotone operators.
\newblock {\em Mathematical Programming}, 55(1):293--318, 1992.

\bibitem{ghadimi2015optimal}
Euhanna Ghadimi, Andr{\'e} Teixeira, Iman Shames, and Mikael Johansson.
\newblock Optimal parameter selection for the alternating direction method of
  multipliers ({ADMM}): quadratic problems.
\newblock {\em IEEE Transactions on Automatic Control}, 60(3):644--658, 2015.

\bibitem{giselsson2017tightglobalrates}
Pontus Giselsson.
\newblock Tight global linear convergence rate bounds for {D}ouglas--{R}achford
  splitting.
\newblock {\em Journal of Fixed Point Theory and Applications},
  19(4):2241--2270, Dec 2017.

\bibitem{giselsson2014diagonal}
Pontus Giselsson and Stephen Boyd.
\newblock Diagonal scaling in {D}ouglas-{R}achford splitting and {ADMM}.
\newblock In {\em Decision and Control (CDC), 2014 IEEE 53rd Annual Conference
  on}, pages 5033--5039. IEEE, 2014.

\bibitem{giselsson2017linear}
Pontus Giselsson and Stephen Boyd.
\newblock Linear convergence and metric selection for {D}ouglas-{R}achford
  splitting and {ADMM}.
\newblock {\em IEEE Transactions on Automatic Control}, 62(2):532--544, 2017.

\bibitem{glowinski2014alternating}
Roland Glowinski.
\newblock On alternating direction methods of multipliers: a historical
  perspective.
\newblock In {\em Modeling, simulation and optimization for science and
  technology}, pages 59--82. Springer, 2014.

\bibitem{golub2013matrix}
Gene~H. Golub and Charles~F. Van~Loan.
\newblock {\em Matrix computations}.
\newblock Johns Hopkins Studies in the Mathematical Sciences. Johns Hopkins
  University Press, Baltimore, MD, fourth edition, 2013.

\bibitem{he2012convergence}
Bingsheng He and Xiaoming Yuan.
\newblock Convergence analysis of primal-dual algorithms for a saddle-point
  problem: {F}rom contraction perspective.
\newblock {\em SIAM Journal on Imaging Sciences}, 5(1):119--149, 2012.

\bibitem{he20121}
Bingsheng He and Xiaoming Yuan.
\newblock On the {$O(1/n)$} convergence rate of the {D}ouglas--{R}achford
  alternating direction method.
\newblock {\em SIAM Journal on Numerical Analysis}, 50(2):700--709, 2012.

\bibitem{he2000alternating}
BS~He, Hai Yang, and SL~Wang.
\newblock Alternating direction method with self-adaptive penalty parameters
  for monotone variational inequalities.
\newblock {\em Journal of Optimization Theory and Applications},
  106(2):337--356, 2000.

\bibitem{li2015proximal}
Xinxin Li and Xiaoming Yuan.
\newblock A proximal strictly contractive {P}eaceman--{R}achford splitting
  method for convex programming with applications to imaging.
\newblock {\em SIAM Journal on Imaging Sciences}, 8(2):1332--1365, 2015.

\bibitem{Liang2017}
Jingwei Liang, Jalal Fadili, and Gabriel Peyr\'e.
\newblock Local convergence properties of {D}ouglas-{R}achford and alternating
  direction method of multipliers.
\newblock {\em J. Optim. Theory Appl.}, 172(3):874--913, 2017.

\bibitem{lin2011linearized}
Zhouchen Lin, Risheng Liu, and Zhixun Su.
\newblock Linearized alternating direction method with adaptive penalty for
  low-rank representation.
\newblock In {\em Advances in neural information processing systems}, pages
  612--620, 2011.

\bibitem{lions1979splitting}
Pierre-Louis Lions and Bertrand Mercier.
\newblock Splitting algorithms for the sum of two nonlinear operators.
\newblock {\em SIAM Journal on Numerical Analysis}, 16(6):964--979, 1979.

\bibitem{moursi2018douglas}
Walaa~M. Moursi and Lieven Vandenberghe.
\newblock Douglas-{R}achford splitting for a {L}ipschitz continuous and a
  strongly monotone operator, 2018.
\newblock arXiv preprint arXiv:1805.09396.

\bibitem{nishihara2015general}
R.~Nishihara, L.~Lessard, B.~Recht, A.~Packard, and M.~Jordan.
\newblock A general analysis of the convergence of {ADMM}.
\newblock {\em arXiv preprint arXiv:1502.02009}, 2015.

\bibitem{oconnor2014primal}
Daniel O'Connor and Lieven Vandenberghe.
\newblock Primal-dual decomposition by operator splitting and applications to
  image deblurring.
\newblock {\em SIAM Journal on Imaging Sciences}, 7(3):1724--1754, 2014.

\bibitem{patrinos2014douglas}
Panagiotis Patrinos, Lorenzo Stella, and Alberto Bemporad.
\newblock Douglas-{R}achford splitting: {C}omplexity estimates and accelerated
  variants.
\newblock In {\em Decision and Control (CDC), 2014 IEEE 53rd Annual Conference
  on}, pages 4234--4239. IEEE, 2014.

\bibitem{pock2011diagonal}
Thomas Pock and Antonin Chambolle.
\newblock Diagonal preconditioning for first order primal-dual algorithms in
  convex optimization.
\newblock In {\em Computer Vision (ICCV), 2011 IEEE International Conference
  on}, pages 1762--1769. IEEE, 2011.

\bibitem{rockafellar1976monotone}
R.~Tyrrell Rockafellar.
\newblock Monotone operators and the proximal point algorithm.
\newblock {\em SIAM Journal on Control and Optimization}, 14(5):877--898, 1976.

\bibitem{rudin1992nonlinear}
Leonid~I. Rudin, Stanley Osher, and Emad Fatemi.
\newblock Nonlinear total variation based noise removal algorithms.
\newblock {\em Physica D: Nonlinear Phenomena}, 60(1-4):259--268, 1992.

\bibitem{song2016fast}
Changkyu Song, Sejong Yoon, and Vladimir Pavlovic.
\newblock Fast {ADMM} algorithm for distributed optimization with adaptive
  penalty.
\newblock In {\em AAAI}, pages 753--759, 2016.

\bibitem{svaiter2018simplified}
Benar~F Svaiter.
\newblock A simplified proof of weak convergence in {D}ouglas-{R}achford method
  to a solution of the unnderlying inclusion problem.
\newblock arXiv preprint arXiv:1809.00967, 2018.

\bibitem{svaiter2011weak}
Benar~Fux Svaiter.
\newblock On weak convergence of the {D}ouglas--{R}achford method.
\newblock {\em SIAM Journal on Control and Optimization}, 49(1):280--287, 2011.

\bibitem{tibshirani1996regression}
Robert Tibshirani.
\newblock Regression shrinkage and selection via the lasso.
\newblock {\em Journal of the Royal Statistical Society. Series B
  (Methodological)}, pages 267--288, 1996.

\bibitem{xu2016adaptive}
Zheng Xu, M{\'a}rio~AT Figueiredo, and Tom Goldstein.
\newblock Adaptive {ADMM} with spectral penalty parameter selection.
\newblock {\em arXiv preprint arXiv:1605.07246}, 2016.

\bibitem{xu2017adaptive}
Zheng Xu, M{\'a}rio~AT Figueiredo, Xiaoming Yuan, Christoph Studer, and Tom
  Goldstein.
\newblock Adaptive relaxed {ADMM}: Convergence theory and practical
  implementation.
\newblock In {\em 2017 IEEE Conference on Computer Vision and Pattern
  Recognition (CVPR)}, pages 7234--7243. IEEE, 2017.

\end{thebibliography}
